\documentclass{article}

\usepackage{arxiv}

\usepackage[utf8]{inputenc} 
\usepackage[T1]{fontenc}    
\usepackage{hyperref}       
\usepackage{url}            
\usepackage{booktabs}       
\usepackage{nicefrac}       
\usepackage{microtype}      
\usepackage{xcolor}         
\usepackage{multirow}
\usepackage{tocloft}        
\usepackage{booktabs} 
\usepackage{placeins}
\usepackage{amsmath,amssymb,amsfonts,amsthm} 
\usepackage{algorithm}
\usepackage{algorithmic}
\usepackage{mathrsfs} 
\usepackage{graphicx}
\usepackage{subcaption}
\usepackage{color}
\usepackage{lastpage}

\newtheorem{theorem}{Theorem}
\newtheorem{lemma}{Lemma}

\newtheorem{corollary}{Corollary}
\newtheorem{remark}{Remark}[section]
\newtheorem{definition}{Definition}[section]
\newtheorem{proposition}{Proposition}[section]
\newtheorem{assumption}{Assumption}[section]

\newcommand{\dd}{\mathrm{d}}
\newcommand{\cL}{\mathcal{L}}
\newcommand{\mM}{\mathcal{M}}
\newcommand{\cN}{\mathcal{N}}
\newcommand{\cR}{\mathcal{R}}
\newcommand{\cS}{\mathcal{S}}
\newcommand{\cT}{\mathcal{T}}
\newcommand{\cF}{\mathcal{F}}
\newcommand{\cO}{\mathcal{O}}
\newcommand{\cP}{\mathcal{P}}
\newcommand{\cM}{\mathcal{M}}

\newcommand{\cB}{\mathcal{B}}
\newcommand{\cE}{\mathcal{E}}
\newcommand{\TM}[1]{\operatorname{Tan}_{#1}\mathcal{M}}

\newcommand{\dx}[2]{\Delta x_{#1}^{(#2)}}
\newcommand{\dr}[2]{\Delta r_{#1}^{(#2)}}

\newcommand{\Id}{\operatorname{Id}}

\newcommand*{\KL}[1][\cdot]{\operatorname{KL}(#1|\rho_\infty)}

\newcommand{\Wspace}{\mathcal{P}_2(\mathbb{R}^d)}
\newcommand{\BWspace}{\cN_0^d}
\newcommand{\TspaceW}[1]{\operatorname{Tan}_{#1}\Wspace}
\newcommand{\TspaceBW}[1]{\operatorname{Tan}_{#1}\BWspace}

\newcommand{\rExp}{\operatorname{Exp}}

\DeclareMathOperator{\Tr}{Tr}

\title{Anderson Mixing in Bures Wasserstein Space of Gaussian Measures}

\author{
 Vitalii Aksenov \\
  WIAS \\
  Berlin, Germany \\
  \texttt{vitalii.aksenov@wias-berlin.de} \\
   \And
 Martin Eigel \\
  WIAS \\
  Berlin, Germany \\
  \texttt{martin.eigel@wias-berlin.de} \\
  \AND
    Mathias Oster \\ 
    RWTH \\
    Aachen, Germany \\
    \texttt{oster@igpm.rwth-aachen.de}
}

\begin{document}
\maketitle

 \begin{abstract}
     Various statistical tasks, including sampling or computing Wasserstein barycenters, can be reformulated as fixed-point problems for operators on probability distributions. 
     Accelerating standard fixed-point iteration schemes provides a promising novel approach to the design of efficient numerical methods for these problems. 
     The Wasserstein geometry on the space of probability measures, although not precisely Riemannian, allows us to define various useful Riemannian notions, such as tangent spaces, exponential maps and parallel transport, motivating the adaptation of Riemannian numerical methods. 
     We demonstrate this by developing and implementing the Riemannian Anderson Mixing (RAM) method for Gaussian distributions. 
     The method reuses the history of the residuals and improves the iteration complexity, and we argue that the additional costs, compared to Picard method, are negligible.  
     We show that certain open balls in the Bures-Wasserstein manifold satisfy the requirements for convergence of RAM.
     The numerical experiments show a significant acceleration compared to a Picard iteration, and performance on par with Riemannian Gradient Descent and Conjugate Gradient methods.

 \end{abstract}

 \section{Introduction}
        There is a plethora of tasks in statistics, ranging from sampling given a Bayesian posterior to calculating  Wasserstein barycenters and medians, that can be reformulated as a fixed-point problem in Wasserstein space. 
        
        For example, the dynamics of the Wasserstein gradient flow of certain functionals is contractive and, thus, induces fix point schemes. 
        To make this more precise, let $\Wspace$ be the space of probability measures with finite second moments
        \begin{equation*}
            \Wspace := \left\{\mu: \mathbb{E}_{x\sim\mu} \|x\|^2_2 < \infty \right\}
        \end{equation*}
        endowed with the 2-Wasserstein distance
        \begin{equation*}
            W^2_2(\mu_1, \mu_2) := \min_{\pi\in \Pi(\mu_1, \mu_2)} \int \|x - y\|^2\mathrm{d}\pi(x,y),
        \end{equation*}
        where $\Pi(\mu_1, \mu_2)$ denotes the set of probability distributions on $\mathbb{R}^d\times\mathbb{R}^d$ with marginals $\mu_1$ and $\mu_2$.

        As shown in~\cite{ambrosio2005gradient}, if some functional $\cE:\Wspace\to\mathbb R \cup\{\infty\}$ is $\lambda$-convex along generalized geodesics for $\lambda>0$, then for any $\mu \in \overline{D(\cE)} := \overline{\{\mu\in \Wspace: \mathcal E(\mu)<\infty\}}$ there exists a unique locally Lipschitz curve $S[\mu](t): (0, +\infty) \to \Wspace$, which is a gradient flow of $\cE$ with initial value $\mu$.
        For a fixed time $\tau$, the operator defined by $G_\tau(\mu) = S[\mu](\tau)$ satisfies the following contraction estimate:
        \begin{equation*}
            W_2(G_\tau(\mu_1), G_\tau(\mu_2)) \leq e^{-\tau \lambda} W_2(\mu_1, \mu_2).
        \end{equation*}
        If $\cE = \operatorname{KL}(\cdot|\rho_*)$ is the Kullback-Leibler divergence and $\rho_* \sim e^{-V}$ for some coercive potential $V$, then $\lambda$-convexity along generalized geodesics is equivalent to $\lambda$-convexity of $V$ and $\rho_*$ is a fixed point of $G$.
        The case of $\operatorname{KL}$ divergence is of particular interest, as a multitude of numerical algorithms for sampling can be interpreted as relaxations and approximations of this gradient flow. 
        Important examples are Langevin dynamics, see, e.g.,~\cite{wibisono2018sampling}, and Stein variational gradient descent, see, e.g., ~\cite{liu2016stein, duncan2023geometry}.
        
        Furthermore, MCMC methods such as Metropolis-Hastings (\cite{hastings1970monte}) or Metropolis-adjusted Langevin dynamics (\cite{roberts1996exponential}) have by design the target distribution as an invariant distribution of the Markov chain.
        Thus, they can also be seen as a fixed-point iteration.
        In addition, fixed-point iterations can be used to find the Wasserstein Barycenter as noted in~\cite{alvarez2016fixed}.
        
        Finally, it is known that for a convex function $f: \mathbb{R}^d \to \mathbb{R}$, a fixed point of $\mathrm{Id}-h\nabla f$ is the root of the gradient and, therefore, the global minimizer.
        Thus, a fixed-point method can be used as a first-order optimization method and can be applied to a variety of optimization problems.

        However, in various instances fixed-point schemes based on a contraction principle suffer from a slow convergence rate, for example if the contraction constant is close to unity. 
        In order to speed up the calculations one can generalize Euclidean or Riemannian acceleration schemes to Wasserstein space by exploiting the manifold-like structure. 
        Previously, similar ideas have been applied in gradient-based optimization.
        In \cite{kent2021modified}, the Frank-Wolfe method is adapted to the Wasserstein setting. 
        A Riemannian Frank-Wolfe method was applied to Wasserstein the barycenter problem in~\cite{weber2023riemannian}. 
        A counterpart of Nesterov acceleration in Wasserstein space is considered in~\cite{liu2018accelerated,chen2025accelerating}.
        Riemannian minimization on the Bures-Wasserstein manifold is studied in~\cite{han2021riemannian}.

        It is known that the space of probability measures on $\mathbb R^d$ 
        can be endowed with the {$2$-W}asserstein metric and then forms a complete separable metric space 
        (\cite{ambrosio2005gradient} Proposition~{7.1.5}).
        This space is not exactly Riemannian (i.e. there is no atlas of maps to subsets of a fixed linear space), but nevertheless all the important geometric notions, such as geodesics, tangent space, exponential mapping and parallel transport, can be defined. 
        However, the deviations of the geometry  from standard Riemannian assumptions, e.g. the exponential map having a vanishing injectivity radius~\cite{gigli2008geometry}, pose hurdles in designing schemes in Wasserstein space beyond the Gaussian case based on already known algorithms. 

    \textbf{Contributions:} 
    We adapt the Riemanian Anderson Mixing (RAM) algorithm introduced in \cite{li2023riemannian} to Gaussian measures in Bures-Wasserstein manifolds (referred to as BWRAM), verifying  some assumptions for the convergence analysis of RAM and identifying the limits of generalizing the methods to general measures in Wasserstein space. 
    The convergence result that we prove can be summarized as follows:
    \begin{theorem}[Convergence of BWRAM (informal statement)]
        Let $G(\Sigma) = \rExp_\Sigma(-F(\Sigma))$ be a contractive mapping on the Bures-Wasserstein manifold and $\Sigma_* \succeq \lambda \Id$ its fixed point with $\lambda>0$. 
        If the initial iterate $\Sigma_0 \in B_{W_2}(\Sigma_*, r)$ for a sufficiently small $r$, and under additional smoothness assumptions on $F$, then the following convergence rate holds
        \begin{equation}
            W_2(\Sigma_k, G(\Sigma_k)) \leq 
            \left[(1-\beta_k) + \kappa\beta_k\right]\|r_k\| + \sum_{i=0}^{\max\{m,k\}}\mathcal{O}(\|r_{k-i}\|^2),  \\
           \label{eq:RAM_local_convergence}
        \end{equation}
        where $\rExp$ denotes the Riemannian exponential map, $r_k = -F(x_k)$ the fixed-point residual, $m$ the maximal number of historical vectors and $\beta_k > 0$ is the relaxation parameter.
    \end{theorem}
    The theoretical estimate says that, up to higher-order terms, the method is guaranteed to converge as good as Picard iteration. 
    We argue that the additional computation costs incurred by the method are negligible in comparison with the operator evaluation, and thus overall speedup of the solution is provided.
    The improved convergence behavior of BWRAM in comparison to Picard is illustrated with various numerical examples such as estimation of the steady-state of the Ornstein-Uhlenbeck process,  minimization of the Kullback-Leibler divergence, and accelerated solution of averaging problems, such as Wasserstein barycenters and geometric medians. 
    We also demonstrate that our method outperforms other well-known Riemannian minimization methods, such as Riemannian Gradient descent and Conjugate Gradient.

    \textbf{Limitations:} 
    The method achieves convergence speed superior to Picard iteration at the expense of additional memory costs and a solution of a small-scale minimization problem.
    The performance of the method is sensitive to the choice of hyperparameters, such as the number of history vectors.
    Although the Picard method is outperformed robustly for an arbitrary selection, the maximal performance is only achieved for particular values, which cannot be known in advance. 
    In $\mathbb{R}^d$, there exist strategies to mitigate the issue, for example, adaptive restarting, see, e.g., \cite{wei2023convergence} and regularization, see, e.g., ~\cite{wei2021stochastic}. 
    Their adaptation is left out to future work. 

    \textbf{Outline:} 
    The paper is structures as follows: In \autoref{sect:Anderson} we recall the Anderson mixing algorithm in Euclidean and Riemannian spaces. 
    Thereafter, we analyze the convergence behavior of the Anderson mixing on Bures-Wasserstein space based on \cite{li2023riemannian} while providing the proofs in the appendix.
    Finally, in \autoref{sect:numeric} numerical experiments are shown.

    The implementation of the method can be accessed at \url{https://github.com/viviaxenov/fpw}

\section{Anderson mixing in Euclidean space and for Riemannian manifolds with bounded sectional curvature}\label{sect:Anderson}
        Anderson mixing (AM) (\cite{fang2009two}) reuses previous calculations in order to accelerate fixed-point iterations. 
        Given the problem of finding $x$ such that $x=G(x)$
        for some operator $G: \mathbb{R}^d \to \mathbb{R}^d$, 
        at each iteration~$k$, AM uses the history of $m$ previous iterates $\{x_{k - i}\}_{i=0}^{m-1}$ and the values of the operator $g_i = G(x_i)$ to choose the next iterate by a linear combination of the histories and current update candidate.
        The coefficients are determined by finding weights $\{v_i\}_{i=1}^{m_k+1}$ such that $\left\|\sum_{i=1}^{m_k+1} v_i r_{k-i}\right\|_2$ is minimized. 
        In a linear space $\mathbb{R}^d$, it takes the form of \autoref{alg:basicAM}.
        \begin{algorithm}
            \caption{Anderson mixing in $\mathbb{R}^d$}
            \label{alg:basicAM}
            \begin{algorithmic}[1]
                \REQUIRE $x_0 \in \mathbb{R}^d$, relaxation parameters $0 \leq \beta_k \leq 1$, and memory parameter $m \geq 1$.
                \ENSURE A sequence $x_0, x_1, \dots$, intended to converge to a fixed-point of $G : \mathbb{R}^d \to \mathbb{R}^d$.
                \FOR{$k = 0, 1, \dots$ until convergence}
                    \STATE Compute $R_k = (r_{k - m}, \dots, r_k)$, where $r_i = G(x_i) - x_i$.
                    \STATE Solve
                    \begin{equation}\label{eq:linear_AA_minimization}
                        \alpha^{(k)} = \underset{v \in \mathbb{R}^{m+1} }{\arg\min}
                            \begin{aligned}
                                &\{\|R_k v\|_2\quad
                                \text{s.t. } \sum_{i=0}^{m} v_i = 1\}.
                            \end{aligned}
                    \end{equation}
                    \STATE Set
                    \begin{equation}\label{eq:linear_AA_update}
                        x_{k+1} = (1 - \beta_k) \sum_{i=0}^{m} \alpha_i^{(k)} x_{k - m + i} + \beta_k \sum_{i=0}^{m} \alpha_i^{(k)} G(x_{k - m + i}).
                    \end{equation}
                \ENDFOR
                \end{algorithmic}
        \end{algorithm}
        Note that the method can be reformulated using residuals $r_k = g_k - x_k$, such that finding the fixed-point is equivalent to setting the residual to zero.
        Introducing a shorthand notation for the forward finite difference
        \begin{equation}
            \Delta r_k = r_{k+1} - r_k,\ \Delta x_k = x_{k+1} - x_k,
        \end{equation}
        and the matrices
        \begin{align}\label{eq:anderson_rd}
            X_k &=[\Delta x_{k-m},\Delta x_{k-m+1},\ldots,\Delta x_{k-1}],\\
            R_k &= [\Delta r_{k-m},\Delta r_{k-m+1},\ldots,\Delta r_{k-1}],\\
            \Gamma &= \left(\gamma^{(k)}_{k - m}, \dots, \gamma^{(k)}_{k - 1} \right)^T,
        \end{align}
        the Anderson mixing in $\mathbb{R}^d$ can be rewritten as
        \begin{equation}
            \left\{
            \begin{aligned}
                & \Gamma_k = \arg\min_\Gamma \left\| r_k - R_k\Gamma \right\|_2,\\
                & \bar{r}_k = r_k - R_k\Gamma_k,\\
                & x_{k+1} = x_k-X_k\Gamma_k+\beta_k \Bar{r}_k.
            \end{aligned}
            \right.
        \end{equation}

        In the Euclidean setting, AM is reported to provide significant numeric advantage in comparison to Picard iteration, see~\cite{aksenov2021application, walker2011anderson} and references therein. Observe that the convergence analysis (\cite{evans2020proof}) in the Euclidean setting yields the estimate
        $$
            \|r_{k+1}\|\leq \frac{\|\bar r_k\|}{\|r_k\|}(1-\beta+\beta\kappa)\|r_k\|+\sum_{i=0}^{m_k}\mathcal O(\|r_{k+i}\|^2).
        $$ 
        where $\frac{\|\bar r_k\|}{\|r_k\|} \leq 1$.
        In the Riemannian setting we can not expect better results.
        As the manifold has no linear structure, differences of the residuals are not well defined. 
        In order to generalize AM to Riemannian manifolds, the definitions of $\Delta r_k$ and $\Delta x_k$ need to be adapted. 
        One needs to transport tangent vectors from the tangent space of previous iterates to the tangent spaces at the current iterate. 
        Therefore, a suitable vector transport has to be defined on the tangent bundle in order to allow linear combinations.
        Additionally, the update needs to be approximated by a retraction mapping to stay on the manifold.
       
        The Riemannian version of Anderson mixing (RAM), incorporating these ideas as introduced in~\cite{li2023riemannian}.
        Therein, the authors work in the setting where the operator takes the form
        \begin{equation}
            G(x) := \rExp_{x}(-F(x))\label{eq:operator_via_residual}
        \end{equation}
        for some vector field $F: x \mapsto \operatorname{Tan}_\cM(x)$.
        One can see that in such a setting that $x_*$ is a fixed-point if and only if $F(x_*) = 0$.
         The method is summarized in \autoref{alg:RAM} with the vector transport $\mathcal{T}_y^x: T_y\mM \to T_x\mM$ and the retraction mapping $\mathcal{R}_x: T_x\mM \to \mM$.
        \begin{algorithm}[ht]
            \caption{Riemannian Anderson mixing method}
            \label{alg:RAM}
            \begin{algorithmic}
                \REQUIRE $x_0 \in\mM, \epsilon, \beta_k > 0, m\in\mathbb{N}^*, k = 1$.
                \STATE $r_0 = - F(x_0)$ and $x_1 = \cR_{x_0}(r_0)$.
                \WHILE{$\|F(x_k)\| \geq \epsilon$}
                \STATE $m_k = \min\{m,k\}$.
                \STATE $\dx{k-i}{k} = \cT_{x_{k-1}}^{x_k}\dx{k-i}{k-1} \in\TM{x_k}, i = 1,\ldots,m_k$.
                \STATE $\dr{k-1}{k} = r_k-\cT_{x_{k-1}}^{x_k}\ r_{k-1}\in\TM{x_k}$.
                \IF{$k \geq 2$}
                \STATE $\dr{k-i}{k} = \cT_{x_{k-1}}^{x_k}\dr{k-i}{k-1} \in\TM{x_k}, i = 2,\ldots,m_k$.
                \ENDIF
                \STATE $X_k = [\dx{k-m_k}{k},\ldots,\dx{k-1}{k}], R_k = [\dr{k-m_k}{k},\ldots,\dr{k-1}{k}]$.
                \STATE $r_k = -F(x_k)$.
                \STATE $\Gamma_k = \arg\min_{\Gamma\in\mathbb{R}^{m_k}}\|r_k-R_k\Gamma\|$.
                \STATE $\Bar{r}_k = r_k-R_k\Gamma_k\in\TM{x_k}, \dx{k}{k} = -X_k\Gamma_k+\beta_k\Bar{r}_k \in\TM{x_k}$.
                \STATE $x_{k+1} = \cR_{x_k}(\dx{k}{k})$, $k = k+1$.
                \ENDWHILE
            \end{algorithmic}
        \end{algorithm}
        
 \section{The Bures-Wasserstein space of Gaussians}
    \subsection{Geometric notions}
    The set of Gaussian measures forms a sub-manifold of the <<manifold>> of probability measures with bounded second moment $\Wspace$ with respect to the Wasserstein metric, see, e.g., \cite{takatsu2010wasserstein,takatsu2011wasserstein, malago2018wasserstein}. The advantage of this restriction is that many geometric notions such as parallel transport have a more computationally feasible formulation in terms of matrix equations in the Gaussian setting than in the general case.

    We define the set of all Gaussians with zero mean as $\BWspace = \{\mathcal{N}(0, \Sigma), 0 \prec \Sigma \in \mathbb{R}^{d\times d}\}$. 
    The set $\BWspace$ can be identified with the set $\operatorname{Sym^{++}(d)}$ of symmetric positive-definite matrices (each measure is identified with its covariance matrix $\Sigma$). 
    Then the Wasserstein distance between two Gaussian measures can be computed as
    \begin{equation}\label{eq:dW_Gaussian}
        W^2_2\left(\mathcal{N}(0, \Sigma_0), \mathcal{N}(0, \Sigma_1) \right) =  \Tr{\Sigma_0} + \Tr{\Sigma_1} 
                - 2\Tr{\left(\Sigma_0^{\nicefrac{1}{2}}\Sigma_1 \Sigma_0^{\nicefrac{1}{2}}\right)^{\nicefrac{1}{2}}}.
    \end{equation}
    We will identify $\mathcal N(0,\Sigma)$ with $\Sigma$ synonymously and use $W_2(\Sigma_0,\Sigma_1)=W_2(\mathcal N(0,\Sigma_0),\mathcal N(0,\Sigma_1)).$
    Given a linear map $ T$, a Gaussian measure with covariance $\Sigma_0$ can be transformed to a new Gaussian with covariance $\Sigma_1=T \Sigma_0 T^*$ by $x\mapsto Tx$.
    Reversely, given two covariances $\Sigma_{0,1}$ there is an optimal linear map $ T$ that transforms the Gaussians into each other. 
    It can be computed as the unique positive solution of the Ricatti equation
    \begin{gather*}
        \Sigma_1 = T\Sigma_0T^*, \qquad
        T = \Sigma_1^{\nicefrac{1}{2}} \left(\Sigma_1^{\nicefrac{1}{2}}\Sigma_0\Sigma_1^{\nicefrac{1}{2}} \right)^{-\nicefrac{1}{2}}\Sigma_1^{\nicefrac{1}{2}} .
    \end{gather*}
    For two measures with parameters $(0, \Sigma_{0,1})$ consider the convex function
    \begin{equation}
        \mathcal{W}(x) = \frac{1}{2}\langle x, Tx \rangle 
    \end{equation}
    where $T$ is the optimal linear map interrelating $\Sigma_0$ and $\Sigma_1$.
    One can see that $(\Id, \nabla W)_\sharp \mathcal{N}(0, \Sigma_0)$ is the optimal transport coupling between $\mathcal N(0,\Sigma_0)$ and $\mathcal N(0,\Sigma_1)$.
    Furthermore, any point on the geodesic between the measures is also a Gaussian: $\rho_t = \mathcal{N}(0, \Sigma_t),\ t\in [0,\ 1]$ and is given by
    \begin{gather}
        \Sigma_t = \left((1 - t)\Id + tT\right)\Sigma_0\left((1 - t)\Id + tT\right) .\label{eq:geodesic_Gauss}
    \end{gather}
    \begin{remark}
        One can also analyze Gaussians with non-zero mean $m_0,m_1$ by adding $\|m_0-m_1\|^2$ to the distance. Then the linear map $T$ becomes affine and $\mathcal W(x)  =\frac{1}{2}\langle x-m_0, T(x-m_0) \rangle + \langle x, m_1\rangle $. The geodesic is augmented by $m_t = (1 - t)m_0 + tm_1 $.
    \end{remark}
    One key ingredient for adapting Anderson mixing to the manifold setting is the notion of tangent spaces and exponential maps. 
    In~\cite{takatsu2010wasserstein}, the tangent space at $\Sigma \in \BWspace$ is identified by $\operatorname{Sym}(d)$ with the scalar product and norm given respectively by
    \begin{equation}
        \langle U, V\rangle_\Sigma := \Tr{U\Sigma V},\quad \|V\|^2_\Sigma = \Tr{V\Sigma V} \label{eq:BW_scalar_prod}.
    \end{equation}
    By the embedding
    \begin{gather*}
        V \in \operatorname{Sym}(d) \mapsto  v(x) = Vx\in L^2_{\mathcal{N}(0, \Sigma_0)}(\mathbb R^d)
    \end{gather*}
    $\operatorname{Sym}(d)$ is identified with a subspace to the general  tangent space  in the Wasserstein space, as defined in \cite{ambrosio2005gradient}. 
        The exponential mapping takes the form
        \begin{equation}
            \rExp_{\Sigma}(V) := (\Id + V)\Sigma(\Id + V).
            \label{eq:rExp_gaussian}
        \end{equation}
Note that the exponential map is only well defined for $V$ with limited norm. For example for $V=-\Id$ the result of the exponential map is not positive definite.

    \subsection{Properties in the neighborhood of a nondegenerate distribution}
    In order to apply local convergence theory of RAM from~\cite{li2023riemannian}, one has to know certain properties of the manifold in the neighborhood of the fixed-point.
    In that regard, we have proven several novel statements, that allow to relate the Bures-Wasserstein distance between the distributions and norms of tangent vectors to Frobenius norms of the symmetric matrices that represent them.
    This significantly simplifies the subsequent numerical analysis.
    We summarize our findings in the following theorem:
    \begin{theorem}\label{them:super_main}
        Let $\Sigma_*$ be a nondegenerate covariance matrix, and 
        \[
            0 < \lambda_d^* \leq \dots \leq \lambda_1^*
        \]
        be the eigenvalues of $\Sigma_*$.
        Define the Bures-Wasserstein ball as
        \begin{equation}
            B_r(\Sigma_*) := \{ \Sigma \in \BWspace: W_2(\Sigma, \Sigma_*) \leq r\}
        \end{equation}
        and consider $r < \frac{1}{2}\sqrt{\lambda_d^*}$.
        
        Then, for $\Sigma, \Sigma_1, \Sigma_2 \in B_r(\Sigma_*)$:
        \begin{itemize}
            \item  $B_r(\Sigma_*)$ is compact in $\cN_0^d$.
            \item Its sectional curvature is uniformly bounded.
            \item The following estimates hold:
                \begin{gather}
                    (\sqrt{\lambda^*_d} - r)\|V\|_F \leq \|V\|_{\Sigma} \leq (\sqrt{\lambda^*_1} + r)\|V\|_F \label{eq:frob_tang}\\
                    \|V\|_{\Sigma_2} \leq \frac{\sqrt{\lambda^*_1} + r}{\sqrt{\lambda^*_d} - r} \| V\|_{\Sigma_1} \label{eq:tang_norm_two_points} \\ 
                    \|\Sigma_1 - \Sigma_2\|_F \leq {2\sqrt{2} (\sqrt{\lambda^*_1} + \sqrt{\lambda^*_d})} W_2(\Sigma_1,\ \Sigma_2) \label{eq:w2_frobenius_bound}
                \end{gather}
            \item         
                For every $V \in \TspaceBW{\Sigma}$ such that 
                    $\|V\|_\Sigma \leq \frac{1}{2}r$, 
                it holds that
                \begin{equation}
                    W_2(\Sigma, \rExp_\Sigma(V)) = \|V\|_\Sigma\ \forall\Sigma\in B_r(\Sigma_*) \label{eq:exponential_injective}               
                \end{equation}
        \end{itemize}
    \end{theorem}
    The proof is split into multiple lemmata, which can be found in the Appendix,~\autoref{app:subsec_properties}.
    The overall idea is to show that a perturbation of the covariance matrix, which is small in terms of the Bures-Wasserstein distance, leads to a small perturbation of the eigenvalues.
    Uniform bound on the eigenvalues leads to bounded sectional curvature due to the explicit formulas, derived in~\cite{takatsu2010wasserstein}.
    Additionally, it allows to compare the tangent norm to the Frobenius norm by~\eqref{eq:frob_tang}.
    As the Frobenius norm does not depend on the current distribution, the dynamic optimal transport can be used to estimate the length of the Bures-Wasserstein geodesic by a Euclidean geodesic with the same endpoints, leading to~\eqref{eq:w2_frobenius_bound}.
    \begin{remark}
        The property~\eqref{eq:exponential_injective}  establishes the injectivity of the exponential map in a neighborhood of zero in the tangent space.
        Together with the sectional curvature bound, these properties are a cornerstone of the convergence analysis of RAM.
        These properties don not hold for the Wasserstein space of arbitrary measures, rendering the convergence analysis for general measures in Wasserstein space difficult. 
    \end{remark}

    \subsection{Construction of the vector transport map}    
    A vector transport mapping has to be defined in order to aggregate the historical vectors to in a common tangent space.
    The requirements on the mapping, sufficient for the convergence of RAM, are given in~\eqref{asmp:vt}.
    Intuitively, the vector transport mapping can be viewed as an approximation of the parallel transport.

    On the space $\Wspace$ the parallel transport is defined in~\cite{ambrosio2008construction} along a class of regular curves.
    The main property of such curves $\rho_t: [0,1] \mapsto \Wspace$, generated by some Lipschitz vector field $v_t$, is that there exist a family of mappings $\tau_s^t$, such that for arbitrary $s,\ t \in [0,1]$
    \begin{equation*}
        \rho_t = (\tau_s^t)_\sharp \rho_s
    \end{equation*}
    and the mappings satisfy the group property, i.e. $\tau_s^p \circ \tau_t^s = \tau_t^p$ and $\tau_s^t = (\tau_t^s)^{-1}$.
    Then, an isometric map called \emph{vector translation} between $L_{2,\rho_t}$ and $L_{2,\rho_s}$ is defined as $u \mapsto u \circ \tau_s^t$
    and the parallel transport is constructed as a limit of the approximation
    \begin{align}
        \mathscr{P}_{t_k}^{t_{k + 1}}u 
            &:= \Pi_{\TspaceW{\rho_{t_{k+1}}}}\left(u \circ \tau_{t_{k+1}}^{t_{k}}\right) \label{eq:vt_one_step_gen},\\
        \mathcal{P}[t_0,\dots t_N]u &:= \mathscr{P}_{t_{N-1}}^{t_N} \circ \mathscr{P}_{t_{N - 2}}^{t_{N-1}} \circ \dots \circ\mathscr{P}_{t_{0}}^{t_1}u \label{eq:pt_definition_gen}
    \end{align}
    with respect to partitions $0 = t_0 < t_1 < \dots < t_N = 1$ as $\max_k|t_{k+1} - t_{k}| \to 0$.
    
    In case of Gaussian measures, the geodesics are regular.
    As shown in~\cite{malago2018wasserstein}, the notion of parallel transport in the sense of~\cite{ambrosio2008construction} then coincides with standard Riemannian parallel transport.
    This mapping is a natural candidate, satisfying this Assumption~\ref{asmp:vt}.
    The ODE for the parallel transport was derived and can be solved numerically with a suitable adaptive ODE solver.
    However, integrating an ordinary differential equation on the manifold can be more numerically challenging than the original fixed-point problem, making parallel transport unsuitable as a component of the algorithm.
    Since the assumptions are quite relaxed, it makes sense to explore other options, trading off the exactness of the approximation with computational efficiency.

    Firstly, we consider an approximation inspired by the discrete scheme~\eqref{eq:pt_definition_gen} with just \emph{one} timestep for the approximation, i.e.
    \begin{equation}
        \cT_{\Sigma_0}^{\Sigma_1} U := \Pi_{\Sigma_1}\left(UT^{-1}_{01} \right),
    \end{equation}
    where $T_{01}$ is the matrix of the optimal transport map between $\Sigma_0$ and $\Sigma_1$.
    \begin{proposition}\label{prop:projection_formula}
        $\Pi_{\Sigma_1}(UT^{-1}_{01})$, the projection to the tangent space at $\Sigma_1$,  is the unique symmetric solution of the equation
        \begin{equation}
            \Sigma_1 X + X\Sigma_1  = \Sigma_1 UT^{-1}_{01} + \left(UT^{-1}_{01} \right)^\intercal\Sigma_1 \label{eq:BW_projection}. 
        \end{equation}
    \end{proposition}
    The proof can be seen in the Appendix,~\autoref{app:derivation_projection_formula}.

    The bound $\|\cT_{\Sigma_0}^{\Sigma_1} U\|_{\Sigma_1} \leq \|U\|_{\Sigma_0}$ for \autoref{asmp:vt} holds because vector translation is an isometry and because of the property of the orthogonal projection.
    \begin{figure}[htbp]
        \centering
        \includegraphics[width=0.55\linewidth]{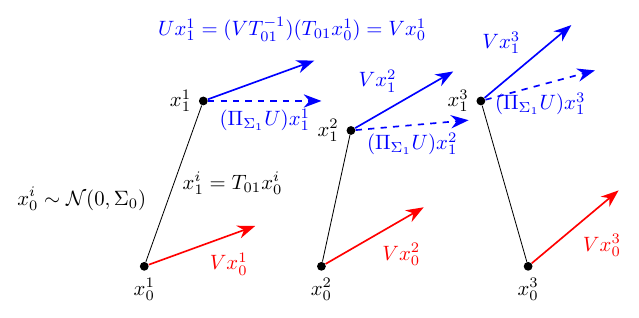}
        \caption{Approximate vector transport along Bures-Wasserstein geodesic }
        \label{fig:vt_gaussian}
    \end{figure}
    The approximation is illustrated in \autoref{fig:vt_gaussian}. 
    Linear map $T_{01}$ pushes the distribution with covariance $\Sigma_0$ forward to the one with covariance $\Sigma_1$.
    The individual particles move in a straight line: $x_1 = T_{01}x_0$.
    The vector translation of the vector field $v(x_0) = Vx_0$ is a composition of this move with $V$, i.e. $U = VT^{-1}_{01}$.
    Note that $U \in \TspaceBW{\Sigma_1}$ if and only if $V$ and $T^{-1}_{01}$ commute, which is not the case in general.
    Thus a projection $\Pi_{\Sigma_1}$ is needed.
    Note that the mapping is well-defined and satisfies the uniform norm bound from~\autoref{asmp:vt} with $M = 1$ for the whole $\BWspace$.
 
    Secondly, in a local neighborhood of a nondegenerate covariance matrix $\Sigma_*$, the identity mapping of $\operatorname{Sym}(\mathbb{R}^d)$ can be viewed as a linear operator between tangent spaces $\Id_{\Sigma_1}^{\Sigma_2}: \TspaceBW{\Sigma_1} \to \TspaceBW{\Sigma_2}$. 
    Due to \eqref{eq:tang_norm_two_points}, it also satisfies the bound
    \[
        \left\|\Id_{\Sigma_1}^{\Sigma_2} V\right\|_{\Sigma_2} \leq   \frac{\sqrt{\lambda^*_1} + r}{\sqrt{\lambda^*_d} - r} \|V \|_{\Sigma_1}\ \forall V \in \TspaceBW{\Sigma_1}
    \]
    which is uniform on the ball $B_r(\Sigma_*)$ with $r < \sqrt{\lambda^*_d}$.
    This is a perfect candidate from the computational perspective (no computation is needed), but the constant may affect the radius of the method's convergence.
    We provide the comparison of the vector transport mappings $\cT_{\Sigma_1}^{\Sigma_2}$ and $\Id_{\Sigma_1}^{\Sigma_2}$ in \autoref{sect:numeric}.

    \begin{remark}

            In this paper, we provide definitions, study the properties and derive numerical algorithms for Riemannian notions such as exponential map, vector transport, projection on the tangent space, etc.
            These can be used to adapt a broad range of Riemannian numerical methods, such as RGD, RCG, RLBFGS etc. to probabilistic tasks, that can be reformulated as fixed-point or optimization problems.
            Although in the convergence analysis we have focused on RAM, in the numerical experiments, we also compare its performance to RGD and RCG.
            Up to our best knowledge, this is the first time these methods are applied in the setting considered.
    \end{remark}

\section{Riemannian Anderson Mixing on Bures-Wasserstein space}
    \subsection{Convergence}
    In order to show the convergence of Anderson mixing in Bures-Wasserstein space we follow the lines of \cite{li2023riemannian}. 
    As we have explicit formulas for the parallel transport $\cP$ as well as the exponential map some arguments of \cite{li2023riemannian} can be simplified. Hence, we repeat the proof in the appendix. 
    As in the Euclidean case in \cite{toth2015convergence,wei2022class}, we impose the <<coercive>> and <<locally Lipschitz continuous>> property on the vector field $F$ as well as on its Jacobian $H$. 
     Define the fixed point map $G(\Sigma):= \rExp_{\Sigma}(-F(\Sigma))$. 
    \begin{assumption} 
        Assume $\Sigma_* \in \mathcal N_0^d$ with $\Sigma_*=G(\Sigma_*)$ and let $\lambda_i^*>0$ be its eigenvalues in decreasing order. 
        Assume there exists a ball $B_r(\Sigma_*)$ with  $r < \frac{1}{2}\sqrt{\lambda^*_d}$ and constants $0 < L_1 < L_2$ and $L_H>0$ such that for all $\Sigma_1,\ \Sigma_2 \in B_r(\Sigma_*)$ holds:\label{asmp:F_lip}
        \begin{gather}
            L_1W_2({\Sigma_1},{\Sigma_2}) \leq \|F({\Sigma_1})-\cP_{\Sigma_2}^{\Sigma_1} F({\Sigma_2})\|_{\Sigma_1} \leq L_2W_2({\Sigma_1},{\Sigma_2}),\\
            \left\| \cP_{\Sigma_2}^{\Sigma_1} H({\Sigma_2}) \cP_{\Sigma_1}^{\Sigma_2} - H({\Sigma_1}) \right\|_{\Sigma_1} \leq L_H W_2({\Sigma_1},{\Sigma_2}), \\
            W_2(G({\Sigma}_1),G({\Sigma}_2)) \leq \kappa W_2({\Sigma}_1,{\Sigma}_2) \text{ for some } \kappa < 1.
        \end{gather}
    \end{assumption}

    \begin{assumption}
        The vector transport $\mathcal T$ is continuously differentiable and there exists some $M>0$ such that 
        \[
            \left\|\mathcal T_{\Sigma_1}^{\Sigma_2} V\right\|_{\Sigma_2}\leq M\|v\|_{\Sigma_1}
        \] 
        for all ${\Sigma_1},{\Sigma_2}\in B_r(\Sigma_*),\ \forall V \in \TspaceBW{\Sigma_1}$.   
        \label{asmp:vt}
    \end{assumption}
Observe that for the vector transports we employ we have $M\leq 1$ and $M\leq \frac{\sqrt\lambda^*_1+r}{\sqrt{\lambda^*_d}-r}$ respectively.
    
    The final assumption concerns the uniform boundedness of the extrapolation coefficients $\Gamma_k$ in the RAM method. This assumption is common for analyzing convergence in the Euclidean setting, see~\cite{toth2015convergence}.  
    \begin{assumption}
        There exists a positive constant $M_{\Gamma}$ such that  $\|\Gamma_k\|_{\infty} \leq M_{\Gamma}$ for all $ k\in\mathbb{N}$.
    \label{asmp:bounded_least_square}
    \end{assumption}
    Following \cite{li2023riemannian}, we can formulate the following theorem.
    \begin{theorem}\label{them:Main}
        Assume $\{\Sigma_n\}_{n\geq 1}$ is generated by \autoref{alg:RAM} and suppose \autoref{asmp:F_lip}-\ref{asmp:bounded_least_square} hold. 
        Additionally, assume that initial point $\Sigma_0\in \mathcal N_0^d$ satisfies $W_2(\Sigma_0,\Sigma_*)\leq \frac{1}{1+L_2}\min\left\{\frac{r}{L_2(mM_1+\beta+1)},\frac{L_1-L_2+(1-\kappa)\beta L_2}{MmL_2^2}\right\}$.
        
        Then for all $k \geq 1$ we have $W_2(\Sigma_k,\Sigma_*)\leq \min\left\{\frac{r}{L_2(mM_1+\beta+1)},\frac{L_1-L_2+(1-\kappa)\beta L_2}{MmL_2^2}\right\}$ and it holds that
        \begin{equation}\label{eq:RAM_conv_est}
            \|r_{k+1}\|\leq \theta_k \left(1-\beta_k+\kappa\beta_k\right)\|r_k\| + \sum_{i=0}^{m_k}\mathcal O(\|r_i\|^2) \quad\text{ where }\quad\theta_k = \frac{\|\bar r_k\|}{\|r_k\|}\leq 1.
        \end{equation}
    \end{theorem}
    \begin{remark}
        The hidden constants in $\sum_{i=0}^{m_k}\mathcal O(\|r_i\|^2)$ take the form $$c_0\tilde rm M_1^2+M_3+2M_\Gamma M_2$$ with $$M_1=M_\Gamma\frac{1}{L_1}\max\{M^{m_k}+M^{m_k-1},M^2+M\}+1.$$
        In contrast, in \cite{li2023riemannian} the constants read $$\left((\kappa+1)(C+\sqrt{K}) + \frac{c_0\tilde r}{2} +\sqrt{K}\max\{1,L_2\}\right)m M_1^2+m\left(M_3+2M_\Gamma M_3 \right)$$ with $$M_1 =\max\{\frac{MM_\Gamma}{L_1 C}+1,M_\Gamma \tilde M,\tilde M \}$$
        where $$\tilde M = \frac 1{L_1 C}\max\{M^{m_k}+M^{m_k-1},M^2+M\} $$
    \end{remark}
    
    Since the mapping is assumed to be contractive, the decrease in residual can be directly linked to the distance to the solution
    \begin{gather*}
        W(\Sigma_k, \Sigma_*) \leq W(G(\Sigma_k), \Sigma_*)) + W(G(\Sigma_k), \Sigma_k)) \leq \kappa W(\Sigma_k, \Sigma_*) + W(G(\Sigma_k), \Sigma_k)) \\
        W(\Sigma_k, \Sigma_*) \leq \frac{1}{1 - \kappa} \| r_k \| 
    \end{gather*}
    The proof can be found in the appendix (\ref{proof:Main}) and follows \cite{li2023riemannian} very closely while improving the constants and simplifying some steps for this particular case.

    The estimate~\eqref{eq:RAM_conv_est} can be interpreted as follows.
    The method decreases the residual at least as good as a relaxed Picard iteration (the residual decrease in that case would be exactly $\left(1-\beta_k+\kappa\beta_k\right)\|r_k\|$), up to ``higher order terms'' $O(\sum_i \|r_{k - i} \|^2)$. 
    When the iterate is sufficiently close to the solution, the higher order terms can be neglected.
    The coefficient $\theta_k = \frac{\|\bar r_k\|}{\|r_k\|}$ quantifies the gain, compared to the Picard step, and depends on the quality of the solution of the subproblem.
    The estimate suggests to keep the history length of the method bounded and relatively short. 
    This intuition is indeed backed by our practical experience, where a significant acceleration can be achieved for quite modest history lengths $m \in {1, \dots, 5}$. 
    As of now, estimates of such type (i.e. including $O$ terms) are a state of the art for the case of Anderson Mixing on nonlinear problems even on the Euclidean space (see \cite{sun2021damped,liu2024anderson,wei2023convergence,evans2020proof}). 

    \subsection{Examples of operators}
    In the following we describe several examples of fixed-point problems, suitable for acceleration with BWRAM.
    We reflect on their contractive properties and and discuss the possible applications.
        \subsubsection{Ornstein-Uhlenbeck process}
        The evolution of the Ornstein-Uhlenbeck process is a suitable model problem, since it is known that the induced fixed-point operator is contractive.
        It is defined by the stochastic differential equation
        \begin{equation}
            \mathrm{d}X_t = -\Sigma_*^{-1}X_t \mathrm{d}t + \sqrt{2}\mathrm{d}W_t
            \label{eq:OU_process}
        \end{equation}
        as follows: if $X_0 \sim \cN(0,\Sigma_0)$, then for some fixed timestep parameter $\tau$,  $G(\Sigma_0)$ is the law of $X_\tau$.
        The process can be identified with the gradient flow of the Kullback-Leibler divergence $\operatorname{KL}(\cdot | \cN(0, \Sigma_*))$ w.r.t. the {$2$-W}asserstein distance. 
        In case of Gaussians, there is an explicit solution~\cite{wibisono2018sampling}
        \begin{equation*}
            X_\tau \sim e^{-\tau \Sigma_*^{-1}}X_0 + \Sigma_*^{\nicefrac{1}{2}}\left(I - e^{-2\tau\Sigma_*^{-1}}\right)^{\nicefrac{1}{2}}Z,
        \end{equation*}
        where $Z \sim \cN(0, I_d)$.
        Thus, in case of $X_0 \sim \cN(0, \Sigma)$, the covariance after one step is given by
        \begin{equation}
            G(\Sigma) = e^{-\tau \Sigma_*^{-1}}\Sigma e^{-\tau \Sigma_*^{-1}} +
                \left(I - e^{-2\tau\Sigma_*^{-1}}\right)^{\nicefrac{1}{2}}\Sigma_*\left(I - e^{-2\tau\Sigma_*^{-1}}\right)^{\nicefrac{1}{2}}. \label{eq:OU_operator_explicit}
        \end{equation}
        For a symmetric positive definite $\Sigma$, the matrices $I,\ \Sigma,\ \Sigma^{-1}, \Sigma^{\nicefrac{1}{2}}$ and $e^{-2\tau \Sigma^{-1}}$ all have the same basis of eigenvectors thus, commute and it is easy to simplify \eqref{eq:OU_operator_explicit}, and verify $\Sigma_*$ is a fixed point.
        \subsubsection{Minimization of functionals}
        Alternatively, the operator can be defined having some vector field $F$ as  $G(\Sigma) = \rExp_\Sigma(-hF(\Sigma))$.
        If $F$ is a Wasserstein gradient of some functional, 
        \[
            F(\rho) = \partial_W \cE(\rho) := \nabla \frac{\delta}{\delta \rho}\cE(\rho)
        \]
        one can opt to minimize this functional with fixed-point iteration for the following operator
        \begin{equation}\label{eq:operator_gd}
            G(\rho) := \rExp_\rho\left(-h\partial_W \cE \right)
        \end{equation}
        This fixed-point iteration can be seen as a generalization of the Gradient Descent method to spaces with Riemannian-like geometry.
        Minimization over the BW space finds applications in Bayesian logistic regression~\cite{katsevich2024approximation,diao2023forward}, distributional robust optimization~\cite{nguyen2023bridging}, etc.
        
        \subsubsection{Kullback-Leibler divergence}
        As a model minimization problem, we consider the Kullback-Leibler divergence:
        \begin{equation}\label{eq:KL_WG}
            \partial_W \left(\operatorname{KL}(\cN(0, \Sigma)|\ \cN(0, \Sigma_*))\right)(x) = \nabla_x \left(\log\rho_\Sigma(x) - \rho_{\Sigma_*}(x) \right) = 
            (\Sigma^{-1} - \Sigma^{-1}_*)x.
        \end{equation}
        The operator can be written as
        \begin{equation}\label{eq:operator_dW_KL}
            G_h(\Sigma) = \rExp_\Sigma\left(-h\partial_W \operatorname{KL}(\Sigma|\Sigma_*) \right) = (\Id - h(\Sigma^{-1} - \Sigma^{-1}_*)) \Sigma (\Id - h(\Sigma^{-1} - \Sigma^{-1}_*)),
        \end{equation}
        where $h > 0$ is a step-size parameter that can be selected so that the iterative procedure converges.

        \begin{theorem}
            For a nondegenerate $\Sigma_*$, there exists a radius $r$ and stepsize parameters $\tau,\ h$, such that the operators, defined by~\eqref{eq:OU_operator_explicit},~\eqref{eq:operator_dW_KL} are contractive on $B_r(\Sigma_*)$
        \end{theorem}
        \begin{proof}
            According to the theory in~\cite{ambrosio2005gradient}, the functional $\operatorname{KL}(\cdot | \cN(0, \Sigma_*))$ is $\nicefrac{1}{\lambda_1^*}$-convex along generalized geodesics.
            Thus, the following contraction estimate for the Ornstein-Uhlenbeck process~\eqref{eq:OU_operator_explicit} can be given:
            \begin{equation}\label{eq:KL_GF_contraction}
                W_2(G(\Sigma_0), G(\Sigma_1)) \leq e^{-\frac{\tau}{\lambda_1^*} } W_2 (\Sigma_0, \Sigma_1).
            \end{equation}

            The operator~\eqref{eq:operator_dW_KL} can be seen as a one-step explicit time discretization of the gradient flow~\eqref{eq:OU_process}.
            If $\Sigma = \Sigma^*$, then $\partial_W\operatorname{KL}(\cN(0, \Sigma)|\cN(0, \Sigma_*)) = 0$ and $\Sigma$ is a fixed-point.
            Thus, it is natural to expect that, for sufficiently small stepsize $h$, the discretization should capture the contractive property of the continuous process.
            The strict proof relies on establishing a property, analogous to the $L$-Lipschitz gradient property of a functional in $\mathbb{R}^d$, in the neighborhood of the fixed-point.
            It is quite technical, and thus is postponed to the Appendix~(\autoref{app:subsec_contraction}).
        \end{proof}

    \subsection{Complexity analysis}\label{subsec:complexity_analysis}
    For fixed-point problems in $\mathbb{R}^d$, the evaluation of $G(x_k)$ typically is the most computationally expensive operation.
    For example, if Anderson Mixing is used in the context of coupled PDE systems, each evaluation of the operator requires a numerical solution of complicated PDE(s), with parameters defined by the current iterate $x_k$~(see \cite{fang2009two,aksenov2021application,lee2022convergence} for more examples). 
    The method's efficiency is evaluated by comparing the amount of the number of the operator calls with Picard method.
    The overhead, induced by the acceleration algorithm can be neglected in this case.
    Indeed, the additional operations consist of finding the mixing coefficients via~\eqref{eq:linear_AA_minimization} and then computing the new iterate by formula~\eqref{eq:linear_AA_update}.
    The storage of historical vectors requires $\cO(m d)$ additional memory.
    The assembly of the matrices, needed to solve~\eqref{eq:linear_AA_minimization} can be estimated as $\cO( m^2d)$ operations.
    Additional $\cO(m^3)$ operations are required to solve the minimization problem.
    Since $m$ is a small number in the range $1\dots 15$, and the complexity of the operator evaluation is expected to scale worse than linear, the neglection of the overhead is justified from the computational complexity perspective, especially for large-scale problems ($d \gg m$).

    If the Gaussian distributions in $\BWspace{}$ are considered, the total number of the degrees of freedom, needed to represent the points and tangent vectors, is $\cO(d^2)$.
    Thus, the storage of the matrices $X_k,\ R_k$ requires $\cO(m d^2)$ memory.
    In addition to the aforementioned steps, Riemannian Anderson Mixing performs the vector transport.
    In case when identity map $\Id_{\Sigma_1}^{\Sigma_2}$ is used, there is no computation at all.
    For the one-step approximation $\cT_{\Sigma_1}^{\Sigma_2}$, a computation of the OT map, and $m$ matrix multiplications and solutions of~\eqref{eq:BW_projection} are needed, thus the total number of iterations can estimated as $\cO(m d^3)$.
    The computation of the matrices in~\eqref{eq:l_inf_subproblem} requires $\cO(m^2 d^3)$, and the update using Riemannian exponential \eqref{eq:rExp_gaussian} costs $\cO(d^3)$.
        
    For example, in the averaging problems, that we consider in the sequel, the computation of each operator requires computing the optimal transport map between the current distribution and $N_\sigma$ other ones.
    This can be efficiently implemented in $\cO(N_\sigma d^3)$ operations using the eigenvalue decomposition. 
    In applications, $N_\sigma$ can be as high as  $10^3$~\cite{chewi2020gradient}. 
    Thus, in our examples the operator computation complexity has the same asymptotic dependence on $d$,  as the overhead introduced by RAM.
    Nevertheless, we argue that it is possible to achieve overall acceleration when the operator is sufficiently ``heavy'', e.g. when $N_\sigma$ is high.
      
   \section{Numerical experiments}\label{sect:numeric}
   The following section provides the implementation details and numeric results.
   The experiments were performed on a laptop with 13th Gen Intel(R) Core(TM) i7-1355U CPU with 12 cores.
   The execution time varied from fractions of second in low-dimensional cases up to $\sim 100\ s$ for $500$ dimensions.
   \subsection{Implementation details }
        Since the explicit form of the exponential is given in \autoref{eq:rExp_gaussian}, we can directly use it as the retraction mapping.

        To adhere to~\autoref{asmp:bounded_least_square}, in the minimization subproblem we impose a bound $l_{\infty,max}$ for the weights $\Gamma_k$, 
            and solve (using \texttt{cvxpy} package~\cite{diamond2016cvxpy}) the regularized problem
        \begin{equation}
            \Gamma_k = 
                \underset{|\Gamma|_\infty \leq l_{\infty,max}}{\arg\min} 
                    \left\|r_k - R_k \Gamma \right\|. 
                    \label{eq:l_inf_subproblem}
        \end{equation}

        \subsection{Ornstein-Uhlenbeck process}

        In the first set of experiments, we accelerate the Picard iteration for the operator~\eqref{eq:OU_operator_explicit}, defined by the Ornstein-Uhlenbeck process~\eqref{eq:OU_process}.
        We vary $d,\ \sigma_\textrm{max}$ and for each pair, we run the Picard iteration and Anderson acceleration with different lengths of history $m$, until the residual norm is smaller than $\epsilon = 10^{-6}$.
        The acceleration measured as $\frac{N_\text{Picard}}{N_\text{AA}}$ is averaged over $n = 6$ independent covariances.
        The maximal mean acceleration, depending on $d$ and $\sigma_\textrm{max}$, is presented in \autoref{tab:OU_acceleration}.
        The improved convergence of BWRAM for the Ornstein-Uhlenbeck problem, in comparison to Picard, RGD and RCG methods, is presented in \autoref{fig:OU_conv_final}.
      \begin{table}[htbp]
        \caption{Maximal mean acceleration for BWRAM, depending on $d$ and $\sigma_{\mathrm{max}}$.}
        \centering
          \begin{minipage}[t]{0.48\textwidth}
                \centering
              \begin{tabular}{lrrrr}
\toprule
$\sigma_{max}$ & $1.0$ & $5.0$ & $10.0$ & $20.0$ \\
$d$ &  &  &  &  \\
\midrule
4 & $1.62$ & $5.89$ & $8.29$ & $14.16$ \\
8 & $1.73$ & $3.68$ & $4.23$ & $5.76$ \\
16 & $1.73$ & $4.00$ & $4.76$ & $5.94$ \\
32 & $1.73$ & $4.12$ & $4.86$ & $5.41$ \\
64 & $1.75$ & $4.33$ & $5.60$ & $6.67$ \\
128 & $1.87$ & $4.60$ & $4.78$ & $5.87$ \\
\bottomrule
\end{tabular}

              \subcaption{Ornstein-Uhlenbeck dynamic}
                \label{tab:OU_acceleration}
          \end{minipage}
          \medskip
          \hfill
          \begin{minipage}[t]{0.48\textwidth}
            \centering
              \begin{tabular}{lrrr}
\toprule
$\sigma_{max}$ & $5.0$ & $10.0$ & $20.0$ \\
$d$ &  &  &  \\
\midrule
4 & $4.09$ & $6.33$ & $10.77$ \\
8 & $3.21$ & $4.22$ & $5.12$ \\
16 & $3.15$ & $4.02$ & $4.89$ \\
32 & $3.13$ & $4.39$ & $4.84$ \\
64 & $3.10$ & $4.61$ & $5.01$ \\
128 & $3.00$ & $4.54$ & $5.20$ \\
\bottomrule
\end{tabular}

              \subcaption{$\operatorname{KL}$ minimization}
            \label{tab:KL_acceleration}
          \end{minipage}
      \end{table}
        \subsection{Minimization of $\operatorname{KL}$ divergence}\label{sect:KL_min}
        
        We can compare the performance of the BWRAM to minimization methods, in particular, Riemannian Gradient Descent (RGD) and Riemannian Conjugate Gradient (RCG)~\cite{absil2009optimization}.
        The numerical implementation of these methods from the package \texttt{pymanopt}~\cite{townsend2016pymanopt} has been used.
        
        The values of $\Sigma_*$ were chosen in the same way, as for the experiments with the Ornstein-Uhlenbeck dynamic.
        The scaling parameter $h = 0.3$ was chosen so that every methods converges.
        Each method is iterated until the Kullback-Leibler divergence is reduced below $\epsilon = 10^{-10}$.
        The acceleration is averaged for different realizations of $\Sigma_*$.
        The best mean acceleration among different history lengths, depending on $d$ and $\sigma_{\max}$, is presented in \autoref{tab:KL_acceleration}.
        The improved convergence of BWRAM for the $\operatorname{KL}$ minimization problem, in comparison to Picard, RGD and RCG methods, is presented in \autoref{fig:KL_conv_final}.
        \begin{figure}[htbp]
            \centering
            \begin{subfigure}[b]{0.49\textwidth}
                \centering
                \includegraphics[width=\textwidth]{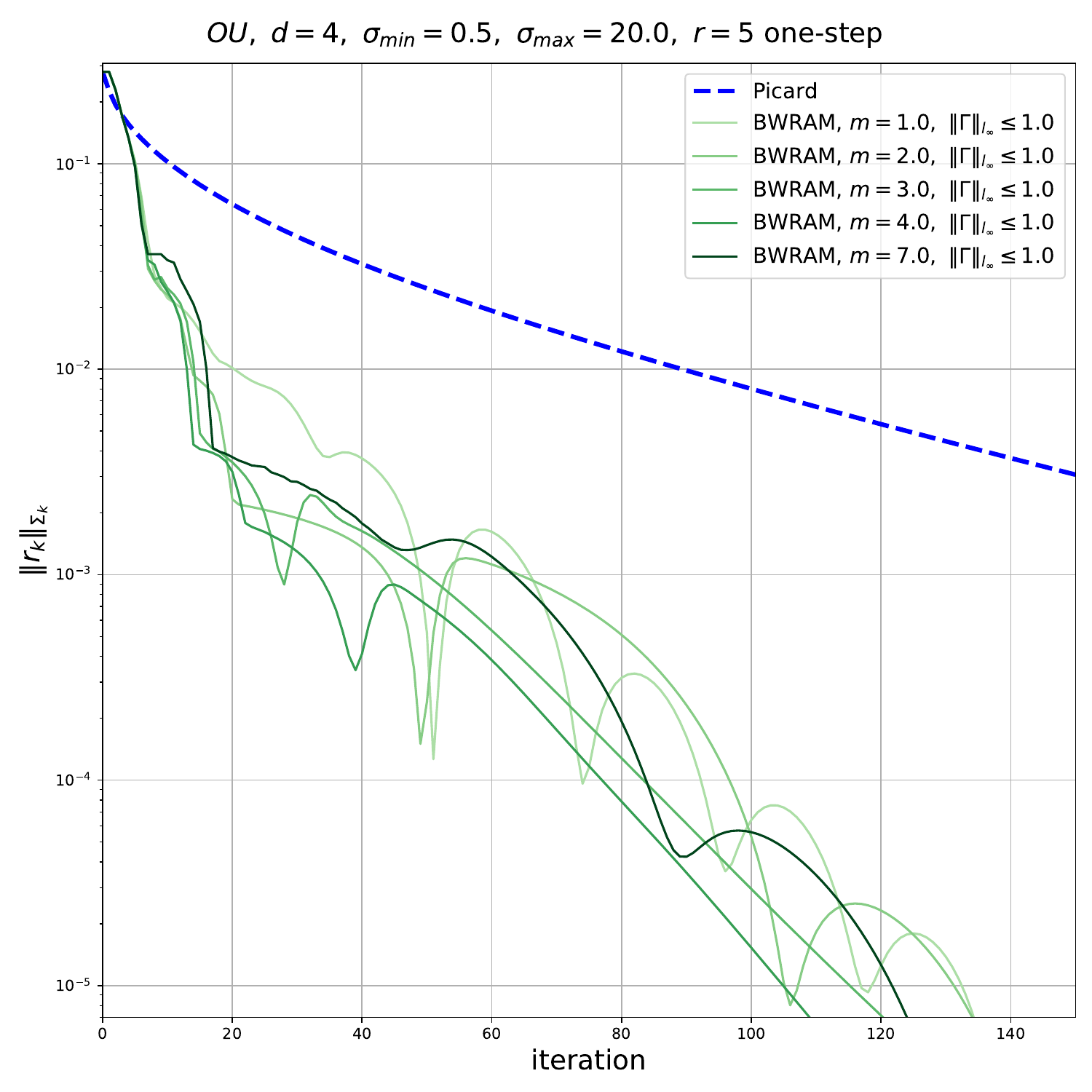}
                \caption{Residual for Ornstein-Uhlenbeck}
                \label{fig:OU_conv_final}
            \end{subfigure}
            \hfill
            \begin{subfigure}[b]{0.49\textwidth}
                \centering
                \includegraphics[width=\textwidth]{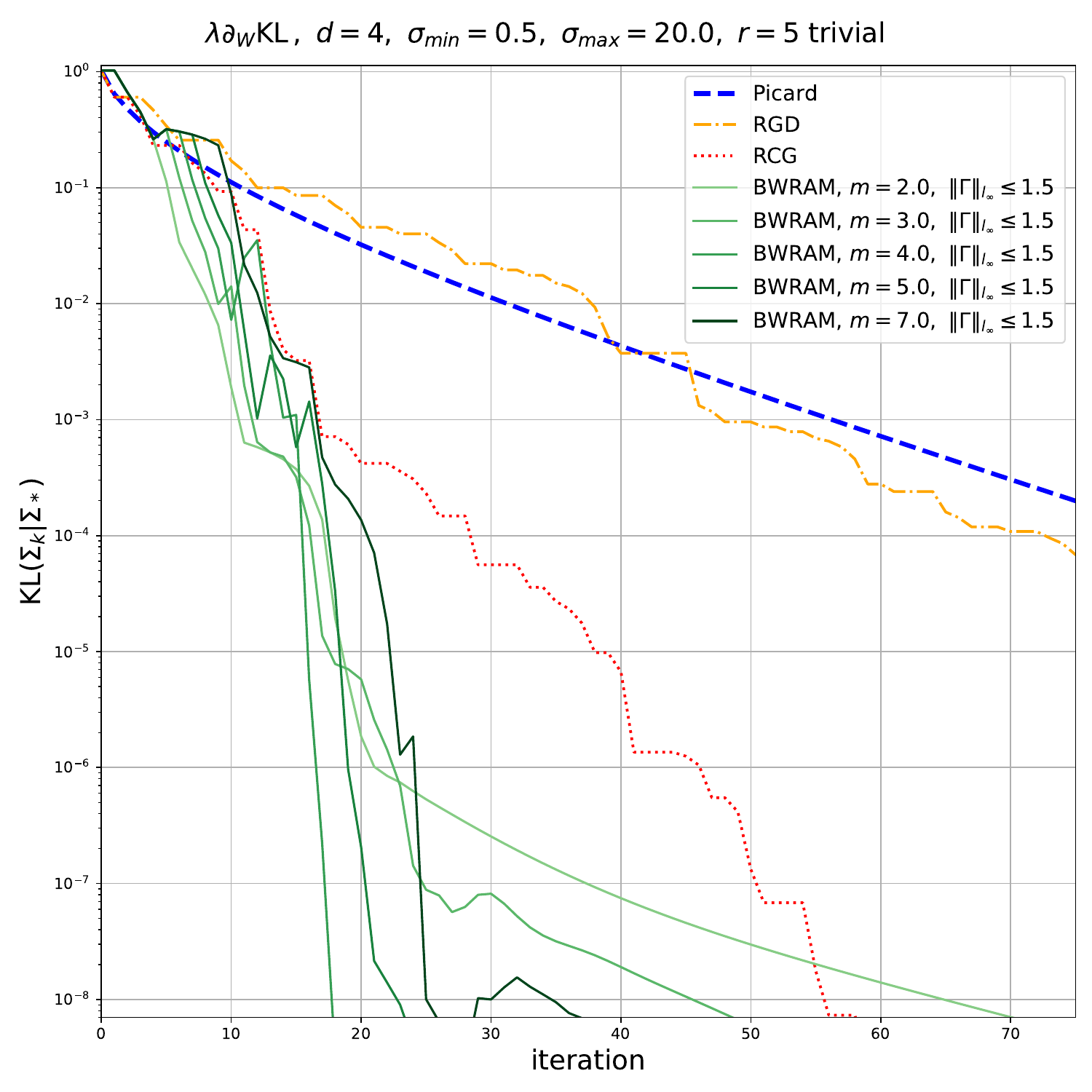}
                \caption{Cost in the $\operatorname{KL}$ problem}
                \label{fig:KL_conv_final}
            \end{subfigure}
            \caption{Convergence of BWRAM on model problems, in comparison to Picard (blue), RGD (orange) and RCG (red) methods }
            \label{fig:Model_conv}
        \end{figure}%
        \subsection{Averaging probability distributions}
        Averaging data sources is an important problem in machine learning (see, for example, \cite{altschuler2021averaging} and references therein).  
        In case of averaging probability distributions, a common approach is to consider a Wasserstein barycenter, introduced in~\cite{agueh2011barycenters}. 
        Given $k$ distributions $\{\rho_i \in \Wspace \}^k_{i=1}$ and weights $\{\alpha_i > 0\ |\ \alpha_1 + \dots + \alpha_k = 1\}$, the Wasserstein barycenter is a minimizer $\min_{\rho \in \Wspace}B(\rho)$  of
        \begin{equation}\label{eq:Barycenter_cost}
            B(\rho) := \sum\limits_{i=1}^{n_\sigma} \alpha_i W^2_2(\rho, \rho_i).
        \end{equation}
        In case of absolutely continuous $\rho$, the Wasserstein gradient of the cost function  is given by~\cite{chewi2020gradient}
        \begin{equation}
            \partial_W B(\rho) = \sum\limits_{i=1}^{n_\sigma}
                \alpha_i (T_{\rho}^{\rho_i} - \operatorname{Id})
                \label{eq:Barycenter_gradient},
        \end{equation}
        where $T_{\rho}^{\rho_i}$ is the optimal transport map from $\rho$ to $\rho_i$.
        The exponential of this gradient defines a fixed-point problem, as noted in~\cite{alvarez2016fixed}.
        The authors also show that if all $\rho_i$ are absolutely continuous and at least one of them has bounded density, then the barycenter is a fixed-point of this mapping. 
        In case that every $\rho_i$ is a Gaussian, the solution is unique.

        Alternative approaches include the \emph{entropy-regularized Wasserstein barycenters} and \emph{Geometric Medians}, which are, respectively, the minimizers of
        \begin{gather}
            B_\gamma(\rho) := \sum\limits_{i=1}^{n_\sigma} \alpha_i W^2_2(\rho, \rho_i) + \gamma \operatorname{KL}(\rho| \cN(0, I_d)) \\
            M(\Sigma) = \sum_{i=0}^{n_\sigma}\alpha_i W_2(\Sigma, \Sigma_i).
        \end{gather}  
        The entropic penalty in the first case can be seen as a way to incorporate prior knowledge in the estimation.
        The geometric median is known to be more robust with respect to the perturbations of the distributions $\rho_k$, but poses additional difficulty to estimate. 
        In particular, $M(\rho)$ 
        is neither geodesically convex nor geodesically smooth.
        Given some smoothing parameter $\varepsilon>0$, we work on a smoothed objective as suggested in~\cite{altschuler2021averaging}, namely
        \begin{equation}
            M_\varepsilon(\rho) = \sum\limits_{i=1}^{n_\sigma} \alpha_i \sqrt{W^2_2(\rho, \rho_i) + \varepsilon^2} .\label{eq:Median_cost}
        \end{equation}

        Constraining to the Bures-Wasserstein manifold $\cN_0^d$, the Wasserstein gradients of these objectives are
        \begin{gather}
            \partial_W B_\gamma(\Sigma) 
                = \sum\limits_{i=1}^{n_\sigma}\alpha_i (T_{\Sigma}^{\Sigma_i} - \operatorname{Id}) + \gamma\left(\Sigma^{-1} - \Id \right),
             \label{eq:EntBC_gradient} \\
            \partial_W M_\varepsilon(\Sigma) 
                = \sum\limits_{i=1}^{n_\sigma}\frac{\alpha_i}{\sqrt{W_2(\Sigma,\Sigma_i)^2 + \varepsilon^2}}\left( T_{\Sigma}^{\Sigma_i} - \operatorname{Id}\right).
            \label{eq:Median_nonregularized_gradient}
        \end{gather}

        The barycenter problem has been computed with Picard iteration~\cite{alvarez2016fixed} or Riemannian gradient descent~\cite{chewi2020gradient}, which coincide if the gradient step size equals one. 
        Stochastic gradient descent has also been considered~\cite{chewi2020gradient,altschuler2021averaging}, with approximation of the sum in~\autoref{eq:Barycenter_gradient} by minibatching.
        For entropic barycenter and median problems, the same methods have been applied e.g. in~\cite{altschuler2021averaging}.
        To the best of the authors' knowledge, this is a first attempt to apply an accelerated method to such problems.

        \begin{figure}[htbp]
            \centering
            \begin{subfigure}[b]{0.32\textwidth}
                \centering
                \includegraphics[width=\textwidth]{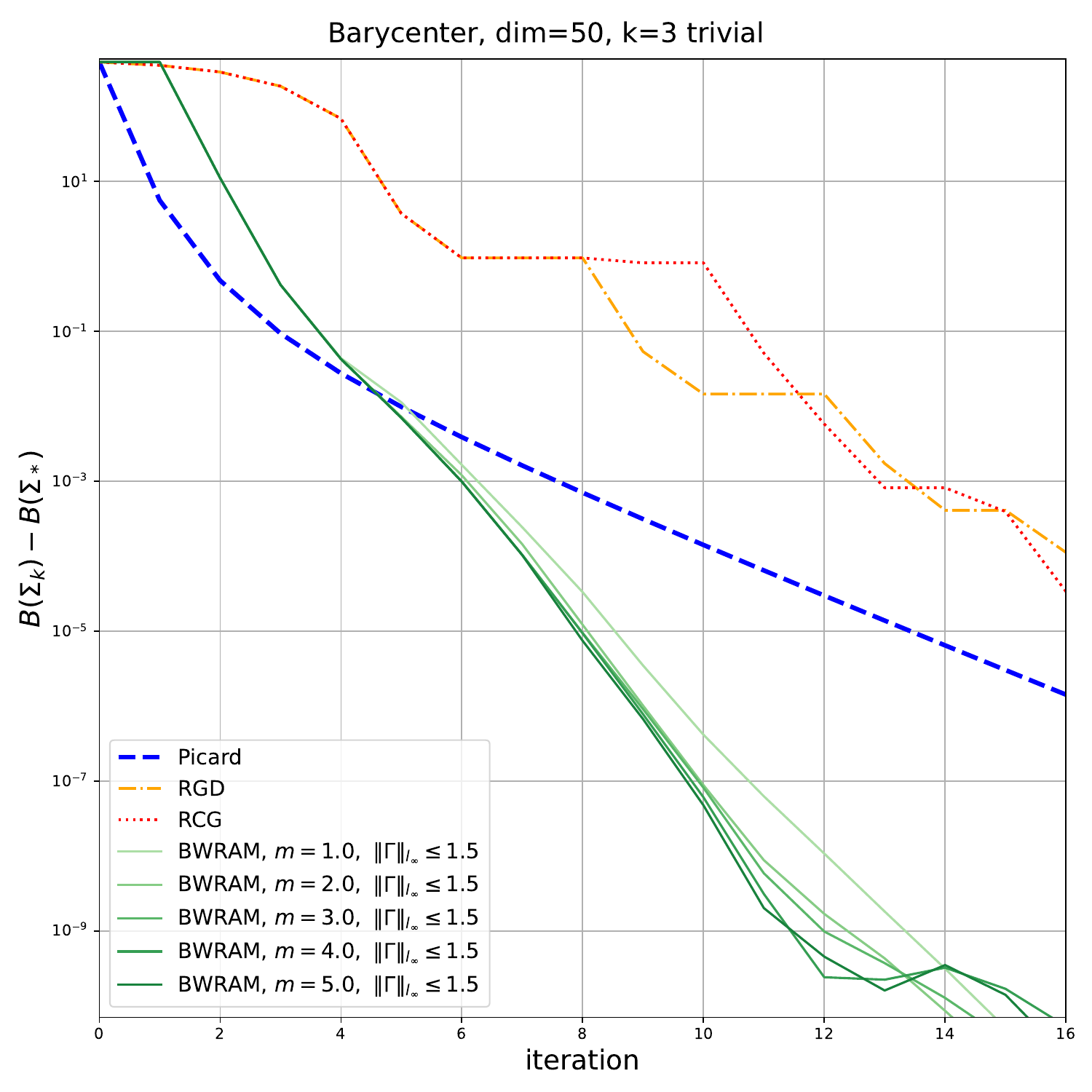}
                \caption{Barycenter}
                \label{fig:BC_conv_final}
            \end{subfigure}
            \begin{subfigure}[b]{0.32\textwidth}
                \centering
                \includegraphics[width=\textwidth]{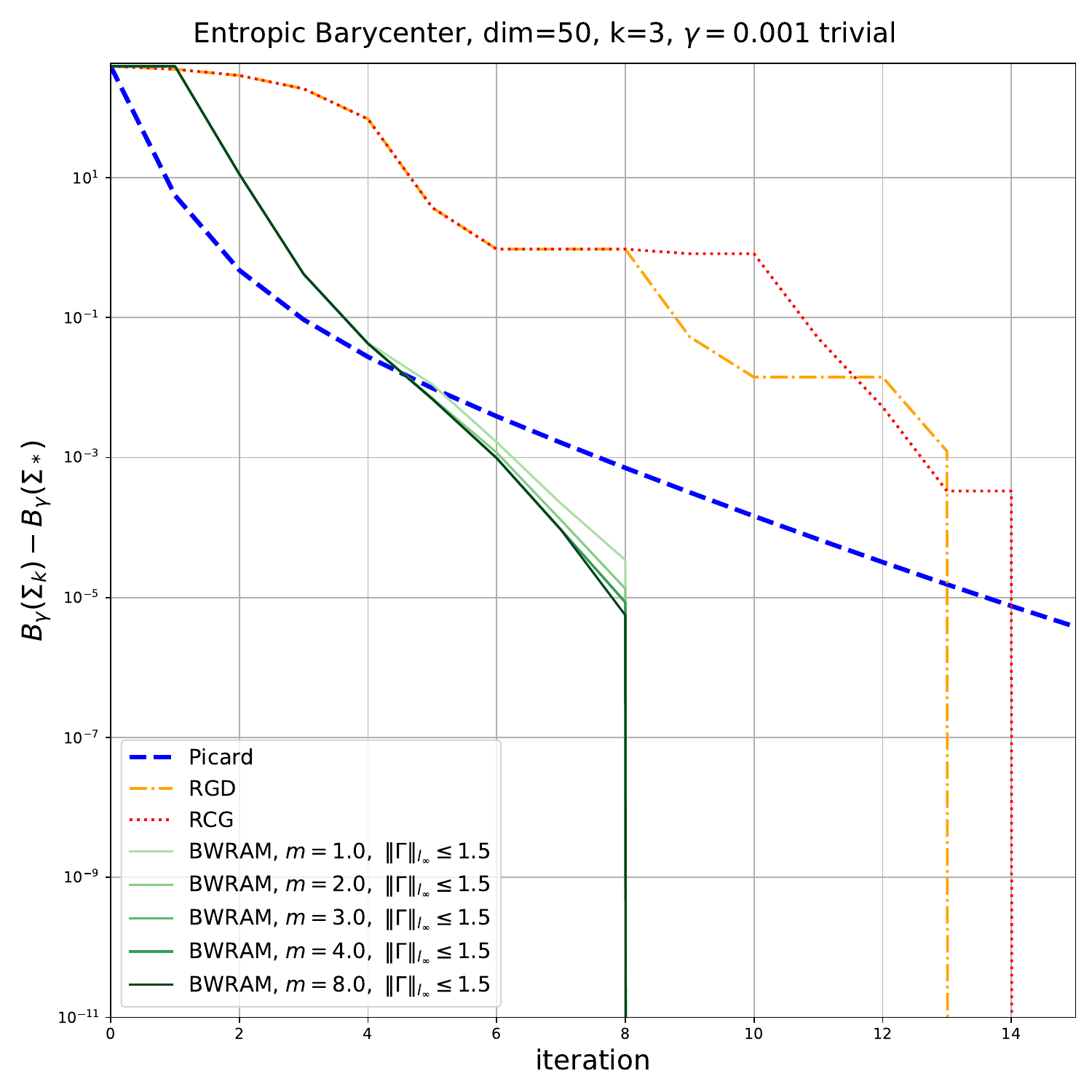}
                \caption{Entropic Barycenter}
                \label{fig:EntBC_convergece_final}
            \end{subfigure}
            \begin{subfigure}[b]{0.32\textwidth}
                \centering
                \includegraphics[width=\textwidth]{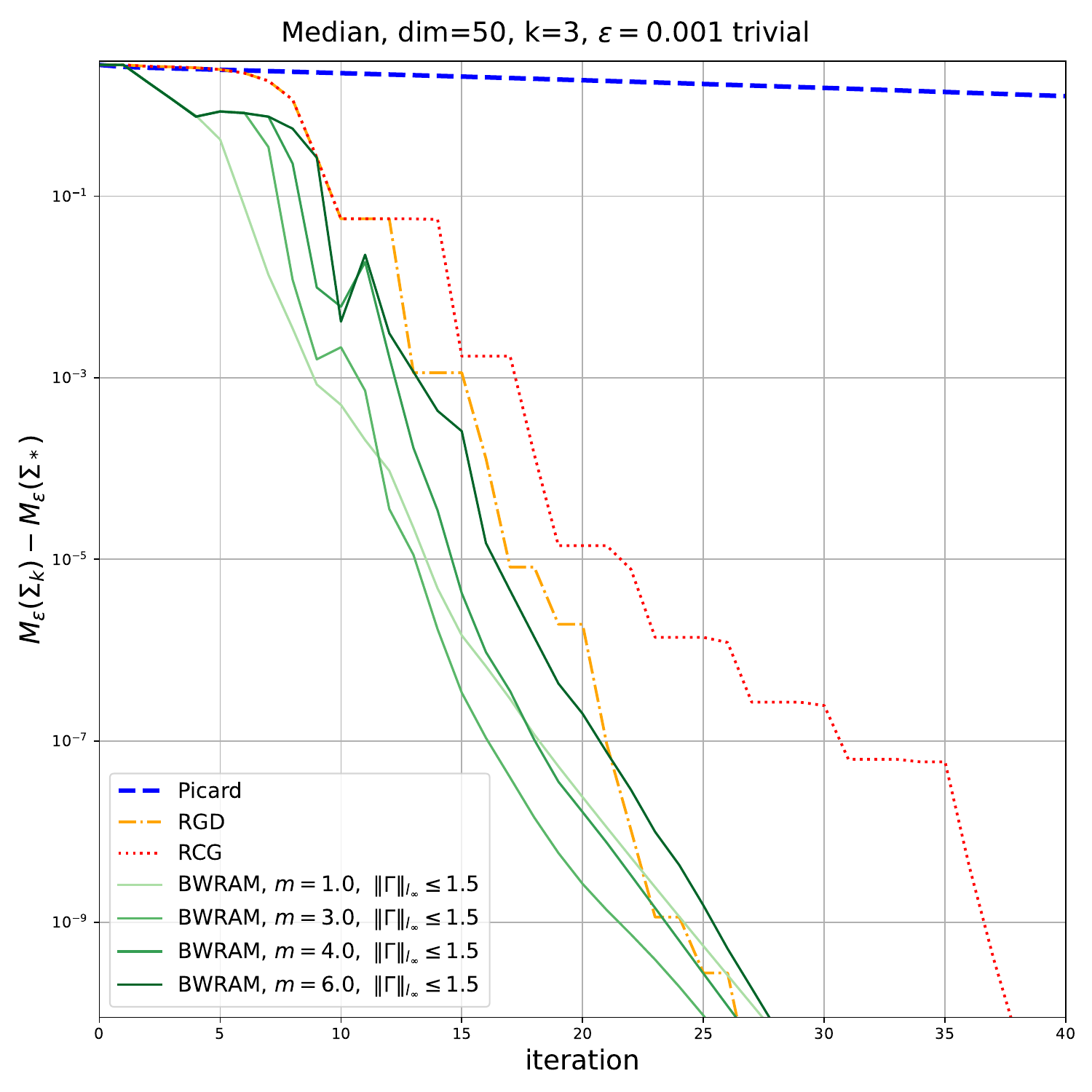}
                \caption{Median}
                \label{fig:Median_conv_final}
            \end{subfigure}
            \caption{Convergence of the cost. BWRAM (greens) with different history length on the averaging problems, in comparison to Picard (blue), RGD (orange) and RCG (red) methods}
            \label{fig:Avg_convergence}
        \end{figure}

        In \autoref{fig:Avg_convergence} one can see that BWRAM outperforms Picard, RGD and RCG for the averaging problems. 
        Details on additional experiments for varying $d, n_{\sigma}$ can be found in the appendix.

        \FloatBarrier
        \subsection{Comparison of the vector transport mappings}
        As noted previously in \autoref{subsec:complexity_analysis}, there are several options for the parallel transport mapping.
        \begin{figure}
            \centering
            \includegraphics[width=1.0\linewidth]{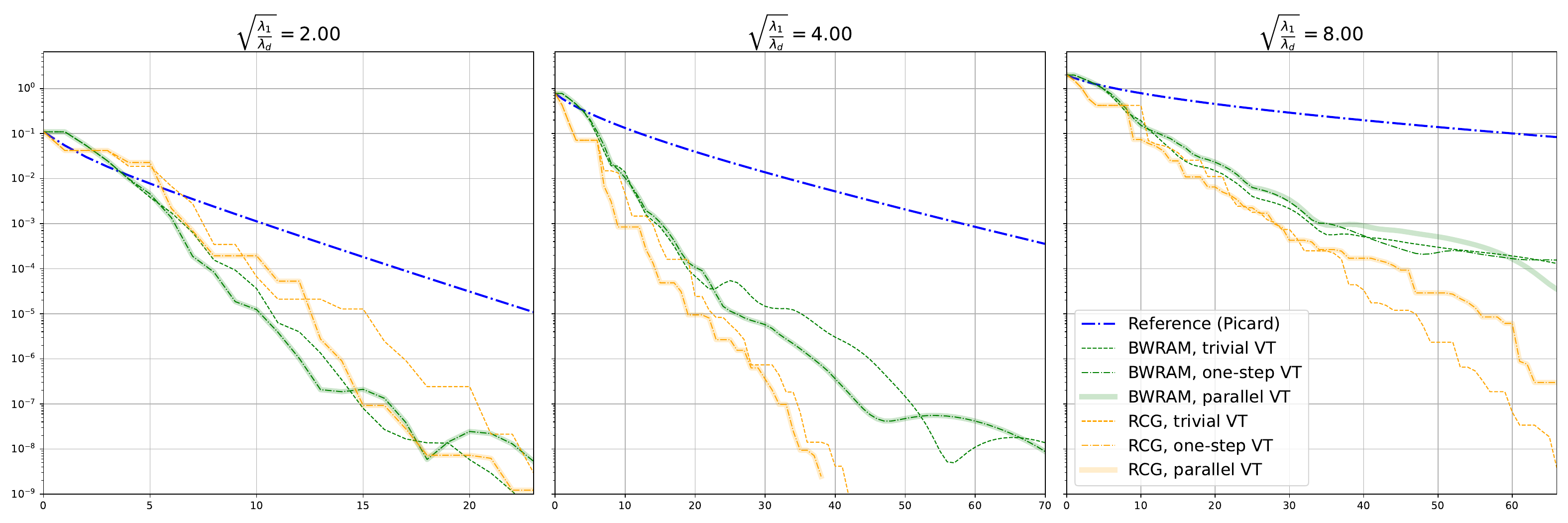}
            \caption{Performance of methods on the $\operatorname{KL}$ problem with varying condition number $\sqrt{\frac{\lambda^*_1}{\lambda^*_d}}$. Exact parallel transport, one-step and trivial approximation.}
            \label{fig:vt_comparison}
        \end{figure}
        The behavior of accelerated methods (BWRAM and RCG) with different kinds of vector transport is represented on~\autoref{fig:vt_comparison}.
        The computation of the exact parallel transport requires a solution of a nonlinear ODE on the manifold.
        This iterative procedure could be comparable in numeric complexity to the solution of the fixed-point problem itself. 
        Moreover, in some cases (e.g. the first two illustration), the resulting trajectory of the method is identical to the one with a one-step approximation.
        These considerations rule out the usage of the exact parallel transport in the practical applications of the method.

        The usage of the trivial transport map does not incur any additional computational complexity, in contrast to the one-step approximation, which can increase the per-iteration complexity.
        On the other hand, in case of trivial transport, the constant $M$ from~\eqref{asmp:vt} depends on the maximal and minimal eigenvalue ratio in the vicinity of the solution, and can be larger than one. 
        This leads to a decrease of the theoretically guaranteed radius of the ball of converging solutions, and to an increase of the constants, hidden in the $\mathcal{O}$ terms of~\autoref{them:Main}.
        Thus, when choosing between one-step and trivial transport map, there exists a tradeoff between the computation efficiency of each iteration and reduced stability of the method and potentially worse acceleration.
        To assess the effect of this tradeoff in the numerical experiments, we perform a robustness study on a subset of our test cases.
        The results are presented in the~\autoref{tab:bwram_vt_comparison}.
        \begin{table}[ht]
            \centering
            \caption{Robustness of the performance with respect to different vector transport mappings. $\operatorname{KL}$ problem with $\sigma_{\text{max}} = 20,\ \sigma_\text{min} = 0.5$ and varying dimension.}
            \begin{tabular}{lrrrr}
\toprule
 & \multicolumn{2}{r}{$\operatorname{Id}_{\Sigma_1}^{\Sigma_2}$} & \multicolumn{2}{r}{$\mathcal{T}_{\Sigma_1}^{\Sigma_2}$} \\
 & $\%$ converged & $\%$ accelerated & $\%$ converged & $\%$ accelerated \\
$d$ &  &  &  &  \\
\midrule
4 & 60.0 & 53.3 & 98.8 & 64.4 \\
8 & 68.0 & 54.0 & 97.3 & 48.4 \\
16 & 63.3 & 53.3 & 94.1 & 53.3 \\
32 & 75.8 & 66.4 & 94.7 & 59.7 \\
64 & 84.4 & 74.7 & 94.9 & 65.3 \\
128 & 74.0 & 66.0 & 96.4 & 66.5 \\
\bottomrule
\end{tabular}

            \medskip
            \label{tab:bwram_vt_comparison}
        \end{table}
        We consider the $\operatorname{KL}$ minimization problem with fixed maximal and minimal covariance, and vary the dimension.
        The Anderson method is repeated with a set of varying parameters over several random initializations of the problem  as described in~\autoref{sect:KL_min}.
        In this set of test runs, we compute the percentage of runs where the method converged to the prescribed tolerance, and where at least double acceleration was achieved (in comparison with the Picard method).
        The outcomes show that the usage of trivial transport mapping does negatively affect the convergence.
        We notice, however, that this is due to divergence of the method on the early stages, which can, in our opinion, be effectively dealt with by incorporating a restarting or early stopping strategy.
        The probability of acceleration is, however, not changed significantly.
        We have also performed experiments, evaluating the maximal mean acceleration and best number of iterations, as in the previous sections, with the both mentioned types of vector transport.
        The results were qualitatively the same, although the particular best values of the hyperparameters do not have to coincide.
        We thus argue that usage of the trivial vector transport is justified, and claim that BWRAM does improve the iteration complexity with a negligible increase of the per-iteration cost.

\section{Conclusion and outlook}
    The numerical evidence suggests that Anderson mixing performs comparably to established Riemannian minimization methods, while being applicable to a broader set of problems outperforming Picard iteration.
    The theoretical analysis certifies the local convergence of the method and highlights the geometric structure of the Bures-Wasserstein space. 
    It also outlines the difficulties that have to be overcome in the future work on generalizing the method to arbitrary distributions.

\section*{Acknowledgement}
VA acknowledges the support of the EMPIR project 20IND04-ATMOC. This project (20IND04 ATMOC) has received funding from the EMPIR programme co-financed by the Participating States and from the European Union’s Horizon 2020 research and innovation programme. ME was supported by ANR-DFG project “COFNET” and DFG SPP 2298
“Theoretical Foundations of Deep Learning”.

MO has received funding by the Deutsche Forschungsgemeinschaft (DFG, German Research Foundation) -- project number 442047500 -- through the Collaborative Research Center ``Sparsity and Singular Structures'' (SFB 1481). 
\section*{Conflict of Interest Statement}
The authors have no conflicts of interest to declare related to this publication.
    \FloatBarrier
\bibliographystyle{unsrt}
\bibliography{references}

\appendix
\section{Omitted proofs}
    \subsection{Properties of $\BWspace$ in a neighborhood of a distribution with a nondegenerate covariance matrix}\label{app:subsec_properties}
    In this section we collect technical lemmata on the local properties of the Bures-Wasserstein manifold.
    In particular, we relate the Bures-Wasserstein distance between the distributions and norms of the tangent vectors to Frobenius norms of the matrices in $\mathrm{Sym}(\mathbb{R}^d)$ that represent them.
    These lemmata are used to prove that the main assumptions on the manifold, needed for the local convergence of RAM, hold in a neighborhood of a nondegegnerate fixed-point (c.f.~\autoref{them:geom_main_assmp}).
    In the following sections, they will also be utilized in order to establish the contractivity of one of the operators in questions, as well as to analyze the properties of vector transport mappings. 
    
    The set $\cN_0^d$ of Gaussian distributions with mean zero and nondegenerate covariance matrix is not complete with respect to the Wasserstein distance.
    But since the Wasserstein distance metricizes weak convergence and it is equivalent to pointwise convergence of the characteristic functions, one can identify~\cite{takatsu2011wasserstein} the closure to the set
    \begin{equation*}
         \overline{\cN_0^d} = \left\{ 
            \mu: \varphi_\mu(\xi) = e^{-\frac{1}{2}\xi^T \Sigma\xi },\quad \Sigma^T = \Sigma \succeq 0 
        \right\}
    \end{equation*}
    where $\varphi_\mu$ denotes the characteristic function of $\mu$.

    Let us denote a Bures-Wasserstein ball with center $\Sigma_*$ and radius $r$, as
    \begin{equation}\label{eq:BW_ball}
        B_r(\Sigma_*) := \{\Sigma \in \BWspace: W_2(\Sigma, \Sigma_*) \leq r \}
    \end{equation}
    In the following, we show that the difference in the eigenvalues of the covariance matrices can be controlled by the Bures-Wasserstein distance between the respective distributions. 
    Thus, the distributions in a small ball $B_r(\Sigma_*)$ around a nondegenerate distribution $\Sigma_*$ have covariance matrices uniformly bound from above and bounded away from zero.
    This allows to prove several technical lemmas, useful for the analysis of the fixed-point operators and the RAM method.

    \begin{lemma}\label{lem:lowerandupperboundsoneigenvalues}
        Let $\rho_0, \rho_1 \in \overline{\cN_0^d}$, with covariance matrices $\Sigma_0 \succ 0,\ \Sigma_1 \succeq 0$.
        If the Wasserstein distance $W^2_2(\rho_0,\ \rho_1) \leq r^2$, then the following bound on the eigenvalues $\lambda^1_1 \geq \dots \lambda^1_d$ of $\Sigma_1$ holds:
        \begin{equation}\label{eq:lambda_W2_bound}
            \sum_k \left(\sqrt{\lambda^0_k} - \sqrt{\lambda^1_k}\right)^2 \leq r^2.
        \end{equation}
        In particular, 
        \begin{align}
            \lambda^1_d &\geq \left(\sqrt{\lambda^0_d} - r\right)^2 \label{eq:lambda_min_W2_bound}, \\
            \lambda^1_1 &\leq \left(\sqrt{\lambda^0_1} + r\right)^2 \label{eq:lambda_max_W2_bound}.    
        \end{align}
    \end{lemma}
    \begin{proof}
        \[
            r^2 = W^2_2(\Sigma_0, \Sigma_1)  %
                = \Tr{\Sigma_0} + \Tr{\Sigma_1} - 2\Tr{\tilde{\Sigma}}%
                = \sum_k \lambda^0_k + \lambda^1_k - 2 \tilde{\lambda}_k
        \]
        where $\tilde{\Sigma} = \left(\Sigma_0^{\nicefrac{1}{2}}\Sigma_1\Sigma_0^{\nicefrac{1}{2}}\right)^{\nicefrac{1}{2}}$ and $\tilde{\lambda}_k$ its eigenvalues.
        Since $\tilde{\Sigma}$ is symmetric, its eigenvalues and singular values correspond.
        Since $\tilde{\Sigma}^T\tilde{\Sigma} = \Sigma_0^{\nicefrac{1}{2}}\Sigma_1\Sigma_0^{\nicefrac{1}{2}} = \left(\Sigma_1^{\nicefrac{1}{2}}\Sigma_0^{\nicefrac{1}{2}}\right)^T\Sigma_1^{\nicefrac{1}{2}}\Sigma_0^{\nicefrac{1}{2}}$, the singular values of $\tilde{\Sigma}$ are the same as the singular values of $\Sigma_1^{\nicefrac{1}{2}}\Sigma_0^{\nicefrac{1}{2}}$.
        As for the singular values of $\Sigma_1^{\nicefrac{1}{2}}\Sigma_0^{\nicefrac{1}{2}}$,
        \[
            \sum_k \sigma_k\left( \Sigma_1^{\nicefrac{1}{2}}\Sigma_0^{\nicefrac{1}{2}}\right) \leq 
                \sum_k \sigma_k\left(\Sigma_0^{\nicefrac{1}{2}}\right)\sigma_k\left(\Sigma_1^{\nicefrac{1}{2}}\right) = \sum_k \sqrt{\lambda^0_k \lambda^1_k}.
        \]
        Thus,
        \[
            r^2 = \sum_k \lambda^0_k + \lambda^1_k - 2 \tilde{\lambda}_k \geq \sum_k \lambda^0_k + \lambda^1_k - 2 \sqrt{\lambda^0_k \lambda^1_k}
            = \left(\sqrt{\lambda^0_k} - \sqrt{\lambda^1_k}\right)^2
        \]
        and \eqref{eq:lambda_min_W2_bound}, \eqref{eq:lambda_max_W2_bound} follow directly.
    \end{proof}

    In the sequel we focus on a ball $B_r(\Sigma_*)$. 
    We denote $0 < \lambda_1^* \leq \dots \leq \lambda_d^*$ the eigenvalues of $\Sigma_*$.
    The radius $r$ is small enough such that $\sqrt{\lambda_1^*} - 2r > 0$.

    The following lemma allows us to relate the norms of the vectors in the different tangent spaces:
    \begin{lemma}
        If $\Sigma, \Sigma_1, \Sigma_2 \in B_r(\Sigma_*)$, $V \in \TspaceBW{\Sigma_1}$, then the following estimates hold
        \begin{gather}
            (\sqrt{\lambda^*_d} - r)\|V\|_F \leq \|V\|_{\Sigma} \leq (\sqrt{\lambda^*_1} + r)\|V\|_F \label{eq:app_frob_tang}\\
            \|V\|_{\Sigma_2} \leq \frac{\sqrt{\lambda^*_1} + r}{\sqrt{\lambda^*_d} - r} \| V\|_{\Sigma_1} \label{eq:app_tang_norm_two_points}
        \end{gather}
    \end{lemma}
    \begin{proof}
        Assuming $A \in \mathbb{R}^{d\times d}, B \in \mathrm{Sym}(\mathbb{R}^d)$, $\lambda_1 \Id \succeq B \succeq \lambda_d \Id$
        \begin{gather*}
            \Tr\left(A^T B A\right) = \sum\limits_i a_i^T B a_i \geq \lambda_d \sum\limits_i a_i^T  a_i \geq \lambda_d\Tr\left(A^TA \right) = \lambda_d \|A\|_F^2 \\
            \Tr\left(A^T B A\right) = \sum\limits_i a_i^T B a_i \leq \lambda_1 \sum\limits_i a_i^T  a_i \leq \lambda_1\Tr\left(A^TA \right) = \lambda_1 \|A\|_F^2
        \end{gather*}
        where $a_i$ are columns of $A$.
        For $\Sigma \in \BWspace$, $V \in \TspaceBW{\Sigma}$, plugging $A \gets V,\ B \gets \Sigma$
        \begin{equation}\label{eq:tang_norm_to_frob_local}
            \lambda_d(\Sigma) \|V\|^2_F \leq \|V\|^2_\Sigma \leq \lambda_1(\Sigma) \| V\|^2_F
        \end{equation}
        Using the uniform bounds \eqref{eq:lambda_min_W2_bound}-\eqref{eq:lambda_max_W2_bound}, one arrives at~\eqref{eq:app_frob_tang}.
        Combining~\eqref{eq:app_frob_tang} twice for $\Sigma=\Sigma_{1, 2}$, one gets~\eqref{eq:app_tang_norm_two_points}.
    \end{proof}
    
    \begin{lemma}\label{lem:frobenius_by_bw}
        Let $\Sigma_1, \Sigma_2 \in B_r(\Sigma_*)$ with $r < \frac{1}{2}\sqrt{\lambda_d^*}$.
        Then the Frobenius norm between the covariance matrices can be controlled by the Bures-Wasserstein distance between the distributions, i.e.
        \begin{equation}\label{eq:app_w2_frobenius_bound}
            \|\Sigma_1 - \Sigma_2\|_F \leq 2\sqrt{2} \left(\sqrt{\lambda^*_1} + \sqrt{\lambda^*_d}\right) W_2(\Sigma_1,\ \Sigma_2)
        \end{equation}
    \end{lemma}
    \begin{proof}
        In the dynamic OT formulation
        \begin{equation}\label{eq:ot_dynamic}
            W^2_2(\Sigma_1, \Sigma_2) 
            = \inf_{\Sigma(\cdot) \in \mathcal{A}}
            \frac{1}{2}\int \left\|V(t) \right\|^2_{\Sigma(t)}\dd t
        \end{equation}
        where $V(t) \in \TspaceBW{\Sigma(t)}$ is a tangent vector to the curve $\Sigma(t)$, and the admissible set of curves $\mathcal{A}$ is defined as
        \begin{equation}
            \mathcal{A} := \left\{\Sigma(t): t\in [0, 1]  
                \left| \substack{
                        \Sigma(t) \in \BWspace \\
                        \Sigma(0) = \Sigma_1 \\
                        \Sigma(1) = \Sigma_2 
                    }  \right. \right\} 
        \end{equation}
        
        First, we show that the infimum in~\eqref{eq:ot_dynamic} wouldn't change if we only consider such curves that $W_2(\Sigma(t), \Sigma_*) \leq 2r \forall t \in [0; 1]$.
        Since $\Sigma_1, \Sigma_2 \in B_r(\Sigma_*)$, from the triangle inequality we argue that
        \[
            W_2(\Sigma_1, \Sigma_2) \leq W_2(\Sigma_1, \Sigma_*) + W_2(\Sigma_*, \Sigma_2) \leq 2r
        \]
        Indeed, let us consider a curve $\Sigma(\cdot)$ and fix some $t\in[0; 1]$.
        We denote  $\Sigma_t := \Sigma(t)$.
        The length of the curve is not smaller than the length of the piecewise geodesic:
        \[
            L(\Sigma(\cdot)) := \sqrt{\frac{1}{2}\int \left\|V(t) \right\|^2_{\Sigma(t)}\dd t} \geq W_2(\Sigma_1, \Sigma_t) + W_2(\Sigma_t, \Sigma_2)
        \]
        Considering the triangle inequality for geodesic triangles $\Delta \Sigma_1 \Sigma_*\Sigma_t$ and $\Delta\Sigma_*\Sigma_t\Sigma_2$, one gets:
        \begin{align*}
            W_2(\Sigma_*, \Sigma_t) \leq W_2(\Sigma_1, \Sigma_t) + W_2(\Sigma_*, \Sigma_1) 
                &\implies  W_2(\Sigma_1, \Sigma_t) \geq W_2(\Sigma_*, \Sigma_t) - W_2(\Sigma_*, \Sigma_1) \\
            W_2(\Sigma_*, \Sigma_t) \leq W_2(\Sigma_t, \Sigma_2) + W_2(\Sigma_2, \Sigma_*)
             &\implies  W_2(\Sigma_t, \Sigma_2) \geq W_2(\Sigma_*, \Sigma_t) - W_2(\Sigma_*, \Sigma_2) \\
        \end{align*}
        Combining with the previous inequality one gets
        \[
            L(\Sigma(\cdot)) 
                \geq W_2(\Sigma_1, \Sigma_t) + W_2(\Sigma_t, \Sigma_2) 
                    \geq 2  W_2(\Sigma_*, \Sigma_t) -  W_2(\Sigma_*, \Sigma_2)- W_2(\Sigma_*, \Sigma_1) > 2W_2(\Sigma_t, \Sigma_*) - 2r
        \]
        If the curve is such that $W_2(\Sigma_t, \Sigma_*) > 2r$, we argue that
        \[
            L(\Sigma(\cdot)) > 4r - 2r = 2r \geq W_2(\Sigma_1, \Sigma_2)
        \]
        and the curve is not optimal.
        Let's denote the improved admissible set of curves by
        \begin{equation}
            \mathcal{A}_r := \left\{\Sigma(t): t\in [0, 1] : 
                \substack{
                        \Sigma(t) \in \BWspace \\
                        W_2(\Sigma(t),\Sigma_*) \leq 2r\  \forall t \in [0; 1]\\
                        \Sigma(0) = \Sigma_1 \\
                        \Sigma(1) = \Sigma_2 
                    }  \right\}
        \end{equation}
        $V(t)$ and $\dot\Sigma(t)$ are related as follows:
        \[
            \Sigma(t + \tau) \underset{\tau \to 0}{\approx} (\Id + \tau V(t))\Sigma(t)(\Id + \tau V(t))
        \]
        thus
        \[
            \dot\Sigma(t) = \Sigma(t)V(t) + V(t)\Sigma(t)
        \]
        Let us define the operator $P_\Sigma$
        \begin{align*}
            P_\Sigma : (\operatorname{Sym}(\mathbb{R}^d),\ \|\cdot\|_F) &\to (\operatorname{Sym}(\mathbb{R}^d),\ \|\cdot\|_F) \\
            P_\Sigma[V] &= \Sigma V + V \Sigma
        \end{align*}
        For $U,\ V \in \operatorname{Sym}(\mathbb{R}^d)$
        \begin{equation*}
            \langle U, P_\Sigma [V] \rangle_F = \Tr\left(U(\Sigma V + V \Sigma) \right) = 2\Tr\left(U\Sigma V\right)
        \end{equation*}
        Thus, $P_\Sigma$ is symmetric.
        On the admissible set of curves $\mathcal{A}_r$, the following bound on its spectrum can be made using~\eqref{eq:app_frob_tang}:
        \[
            2(\sqrt{\lambda^*_d} - 2r)^2 \| V\|^2_F \leq \langle V, P_\Sigma [V]\rangle_F \leq 2(\sqrt{\lambda^*_1} + 2r)^2
        \]
        Thus, $P_\Sigma$ is invertible.
        We define $L_\Sigma := P_\Sigma^{-1}$.
        For $L_\Sigma$ it holds
        \begin{equation}
            \langle U, L_\Sigma[U] \rangle_F \geq \frac{1}{2(\sqrt{\lambda^*_1} + 2r)^2} \|U\|^2_F
        \end{equation}

        For our curve $\Sigma(t)$, since $V(t)$ is tangent, we have
        \begin{align*}
            \dot\Sigma(t) = P_{\Sigma(t)}[V(t)] \\
            V(t) = L_{\Sigma(t)}[\dot\Sigma(t)]
        \end{align*}
        Then (omitting time dependence for brevity)
        \begin{multline*}
            \|V\|^2_\Sigma 
            = \Tr\left(V\Sigma V\right)
            = \frac{1}{2}\Tr\left(V(\Sigma V + V\Sigma\right)
            = \frac{1}{2}\Tr\left(V\dot\Sigma)\right) = \\
            = \frac{1}{2}\Tr\left(L_\Sigma[\dot\Sigma] \dot\Sigma \right)
            = \frac{1}{2}\langle \dot\Sigma, L_{\Sigma}[\dot\Sigma]\rangle_F
            \geq \frac{1}{2(\sqrt{\lambda^*_1} + 2r)^2} \|\dot\Sigma\|^2_F
        \end{multline*}
        where the final inequality holds only for the curves in $\mathcal{A}_r$.
        
        Combining the estimates, we arrive at
        \begin{multline*}
            W^2_2(\Sigma_1, \Sigma_2) 
            = \inf_{\mathcal{A}}
                \frac12 \int \left\| V(t) \right\|^2_{\Sigma(t)} \dd t = \\
            = \inf_{\mathcal{A}_r}
                \frac12 \int \left\| V(t) \right\|^2_{\Sigma(t)} \dd t
            = \inf_{\mathcal{A}_r}
            \frac{1}{4}\int\Tr\left(\mathop{L_{\Sigma(t)}}[\dot\Sigma(t)]\dot\Sigma(t) \right) \dd t \geq \\ 
            \overset{(*)}{\geq}\frac{1}{8 (\sqrt{\lambda^*_1} +  \sqrt{\lambda^*_d})^2}\inf_{\substack{
                            \Sigma(t) \in \BWspace \\
                            \Sigma(0) = \Sigma_1 \\
                            \Sigma(1) = \Sigma_2 
                        }}
            \int \left\| \dot\Sigma(t)\right\|_F^2 \dd t 
            = \frac{\|\Sigma_1 - \Sigma_2\|_F^2}{8 (\sqrt{\lambda^*_1} + \sqrt{\lambda^*_d})^2}
        \end{multline*}
        In~${(*)}$,  we can minimize with respect to the original admissible set $\mathcal{A}$, since $\mathcal{A}_r \subset \mathcal{A}$ and thus
        $\inf_\mathcal{A} \cF \leq \inf_{\mathcal{A}_r} \cL$ for any functional $\cL$.
        \[
            \Sigma_e(t) := (1 - t)\Sigma_1 + t\Sigma_2  
        \]
        and $\sqrt{\lambda^*_d} < 2r$.
        Thus 
        \begin{equation*}
            \|\Sigma_1 - \Sigma_2\|_F \leq {2\sqrt{2} (\sqrt{\lambda^*_1} + \sqrt{\lambda^*_d})} W_2(\Sigma_1,\ \Sigma_2)
        \end{equation*}
    \end{proof}

    The explicit expression for the sectional curvature of the Bures-Wasserstein space is given in~\cite[Theorem {$1.1$}]{takatsu2010wasserstein}.
    Its minimal and maximal values depend only on the minimal and maximal eigenvalues of the covariance matrix.
    We express this fact the following corollary
    \begin{corollary}\label{cor:sec_curve}
        For every $\Sigma \in B_r(\Sigma_*)$, the sectional curvature $K_\Sigma$ of $\BWspace$ at $\Sigma$ can uniformly bound by
        \begin{equation}
            0 \leq K_{\Sigma}(e_1, e_2) \leq \frac{3 (\sqrt{\lambda^*_1} + r)^4}{2(\sqrt{\lambda^*_d} - r)^6}.
        \end{equation}
        where $e_1,\ e_2$ are the directions, defining the plane and $r < \sqrt{\lambda_d^*}$
    \end{corollary}

    We can finally formulate the theorem showing that the main assumptions of the RAM's convergence hold for $B_r(\Sigma_*)$ (analogous to Assumption~$1$ of \cite{li2023riemannian}).
    \begin{theorem}\label{them:geom_main_assmp}
        Let $\Sigma_*$ be a nondegenerate covariance matrix  and $r < \frac{2}{3}\sqrt{\lambda_d^*}$.
        The ball $B_r(\Sigma_*)$ is compact in $\cN_0^d$ with bounded sectional curvature.
        For every $V \in \TspaceBW{\Sigma}$ such that 
            $\|V\|_\Sigma \leq \frac{1}{2}r$, 
        it holds that
            $W_2(\Sigma, \rExp_\Sigma(V)) = \|V\|_\Sigma$ for all $\Sigma\in \cM$.
    \end{theorem}
    \begin{proof}[Proof of \autoref{them:geom_main_assmp}]
    Closedness of the ball and the sectional curvature bounds are due to~\autoref{lem:lowerandupperboundsoneigenvalues} and~\autoref{cor:sec_curve}.

    Every tangent vector $V$ induces a mapping $T(x) = Vx + x = (\Id + V)x = \nabla_x \frac{1}{2}\langle x, (\Id +V)\rangle$.
    Using~\eqref{eq:app_frob_tang}, we get
    \[
        \|V\|_o 
        \leq \| V\|_F 
        \leq \frac{1}{\sqrt{\lambda_*^d} - r} \|V\|_\Sigma
        \leq \frac{r}{2(\sqrt{\lambda_*^d} - r)} = \frac{\nicefrac{2}{3}}{2(1 - \nicefrac{2}{3})} = 1
    \]
    Since $\| V\|_o \leq 1$, operator $(\Id + V)$ has nonnegative eigenvalues, and the mapping $T$ is a gradient of a convex function.
    Thus, the mapping is optimal, and its $L_2$ norm is equal to the Wasserstein distance between the distributions.
    \end{proof}

    \subsection{Contraction coefficients of the operators}\label{app:subsec_contraction}
    For the operators of the form 
    \begin{equation}\label{eq:operator_gd_step}
        G(\rho) = \rExp_{\rho}(-h\partial_W \cE(\rho))
    \end{equation}
    the contraction estimate can be deduced as follows.
    Consider random variables $x^\prime \sim G(\rho_1),\ y^\prime \sim G(\rho_2)$, and a coupling
    \begin{align*}
        x^\prime &= x - h\partial_W \cE(\rho_1, x) \\
        y^\prime &= y - h\partial_W \cE(\rho_2, y) \\
        (x,\ y) &\sim \gamma_{\text{opt}}
    \end{align*}
    where $\gamma_{\text{opt}}$ is an optimal plan between $\rho_1$ and $\rho_2$.
    This coupling is a valid coupling between $G(\rho_1)$ and $G(\rho_2)$.
    Thus
    \begin{multline}
        W^2_2(G(\rho_1), G(\rho_2)) \leq \int \|x^\prime - y^\prime \|^2_2\dd \gamma_{\text{opt}} = 
        \int \|x - y - h( \partial_W \cE(\rho_1, x) - \partial_W \cE(\rho_2, y)) \|^2\dd \gamma_{\text{opt}} = \\ =
            \int \|x - y\|^2\dd \gamma_{\text{opt}}
            - 2h\int \langle \partial_W \cE(\rho_1, x) - \partial_W \cE(\rho_2, y), x - y\rangle \dd \gamma_{\text{opt}}
            + h^2 \int \|\partial_W \cE(\rho_1, x) - \partial_W \cE(\rho_2, y) \|^2  \dd \gamma_{\text{opt}} \label{eq:contraction_intermediate}
    \end{multline}
    If the functional $\cE$ is $\lambda$-convex a.g.g., then the second term can be estimated as follows~\cite[Equation {10.1.8}]{ambrosio2005gradient}:
    \begin{equation}\label{eq:lambda_convex_fn}
            \int \langle \partial_W \cE(\rho_1, x) - \partial_W \cE(\rho_2, y), x - y\rangle \dd \gamma_{\text{opt}} \geq \frac{1}{2}\lambda W^2_2(\rho_1,\rho_2)
    \end{equation}
    To estimate the final term, we can define the generalization of the Lipschitz gradient property in the following way.
    \begin{definition}
        A functional $\cE: \Wspace \to \mathbb{R}$ has $L$-Lipschitz Wasserstein gradient for some subset $\cS \subseteq \Wspace$, if its Wasserstein gradient 
        \[
            \partial_W \cE(\rho) = \nabla \frac{\delta \cE}{\delta \rho} \in L_2(\rho)\ \forall \rho \in \cS
        \]
        and there exists a constant $L > 0$ such that $\forall \rho_1,\ \rho_2 \in \cS$
        \begin{equation}\label{eq:lipschitz_wg}
            \int \|\partial_W \cE(\rho_1, x) - \partial_W \cE(\rho_2, y) \|^2  \dd \gamma_{\text{opt}} \leq L^2 W^2_2(\rho_1, \rho_2)
        \end{equation}
    \end{definition}
    \begin{proposition}
        If the functional $\cE$ is $\lambda$-convex along generalized geodesics and has $L$-Lipschitz Wasserstein gradient, then the operator
        $\rExp_\rho(-h\partial_W\cE)$ is contractive for small enough $h$.
    \end{proposition}
    \begin{proof}
        Plugging~\eqref{eq:lambda_convex_fn} and~\eqref{eq:lipschitz_wg} into~\eqref{eq:contraction_intermediate}, one gets 
        \begin{equation*}
            W^2_2(G(\rho_1), G(\rho_2)) \leq (1 - \lambda h + h^2L^2) W^2_2(\rho_1, \rho_2)
        \end{equation*}
        The coefficient in front of~$W^2_2(\rho_1, \rho_2)$ is smaller than one if  $h \in (0, \frac{\lambda}{L^2})$.
        In particular, the optimal step is $h = \frac{\lambda}{2L^2}$.
    \end{proof}

    In~\cite{ambrosio2005gradient}, there are multiple examples of $\lambda$-convex a.g.g. functionals.
    In particular, for the functional $\cE = \KL$, where $\rho_\infty = e^{-V}$,  $\cE$ is $\lambda$-convex when $V$ is $\lambda$-convex.
    The Lipschitz gradient property, up to the authors' best knowledge, hasn't been thoroughly studied before.
    In the following, we perform this analysis in the case of the  $\KL$ functional, constrained to the Bures-Wasserstein manifold.
    lets compute in case of Gaussians, $\rho_i = \cN(0, \Sigma_i),\ i= \overline{1,2}$.
    \begin{lemma}
        Let $\Sigma_*$ be a nondegenerate covariance matrix.
        The functional 
        \[
            \operatorname{KL}(\cdot | \Sigma_*):\ \BWspace \to \mathbb{R}^d
        \]
        has $L$-Lipschitz Wasserstein gradient on a ball $B_r(\Sigma_*)$ where
        \begin{equation}
            r < \frac{1}{16\sqrt{2}\left( 1 + \sqrt{\frac{\lambda^*_1}{\lambda^*_d}}\right)}\sqrt{\lambda_d^*}
        \end{equation}
        and 
        \begin{equation*}
            L^2 = \frac{4}{(\sqrt{\lambda^*_d} - r)^4} \left( 1 + 4\frac{{(\sqrt{\lambda^*_1} + 2r)^4} }{(\sqrt{\lambda^*_d} - r)^{4}} \right)
        \end{equation*}
    \end{lemma}
    \begin{proof}
    The idea of the proof is to work with the explicit form of~\eqref{eq:contraction_intermediate} and split the quantity under the integral w.r.t. $\mathrm{d}\gamma_{\text{opt}}$ in two parts.
    The first is ``linear part'' in form of $\|A(x - y)\|$ for some operator $A$, which can be directly related to the Bures-Wasserstein distance. 
    The second term is a ``nonlinear part'' $\|\Sigma^{-1}_1 -  \Sigma^{-1}_2\|_F$, which we treat by first considering the linearization w.r.t. $\|\Sigma_1 - \Sigma_2\|_F$, and then relating the norm difference to the Wasserstein distance using~\eqref{eq:w2_frobenius_bound}.

    Assuming that $\Sigma_1,\ \Sigma_2 \in B_{r}$ with $r < \frac{1}{2}\sqrt{\lambda_d^*}$, so that the previous lemmata can be used, we proceed with the calculation:
    \begin{multline}\label{eq:lipschitz_gradient_start}
        \int \|\partial_W \cE(\Sigma_1, x) - \partial_W \cE(\Sigma_2, y) \|^2  \dd \gamma_{\text{opt}} = \\ =
        \int \left\|(\Sigma_1^{-1} - \Sigma_*^{-1})x - (\Sigma_2^{-1} - \Sigma_*^{-1})y\right\|^2 \dd \gamma_{\text{opt}}=
        \int \left\|(\Sigma_1^{-1} - \Sigma_*^{-1})(x - y) + (\Sigma_1^{-1} - \Sigma_2^{-1})y\right\|^2 \dd \gamma_{\text{opt}} \leq \\ \leq
        2\|\Sigma_1^{-1} - \Sigma_*^{-1}\|_o^2 \int \|x - y\|^2 \dd \gamma_{\text{opt}} +
            2\int \left\|(\Sigma_1^{-1} - \Sigma_2^{-1})y\right\|^2 \dd \gamma_{\text{opt}} = \\ =
        2\|\Sigma_1^{-1} - \Sigma_*^{-1}\|_o^2 W^2_2(\Sigma_1, \Sigma_2) 
            + 2\left\|\Sigma_2^{\nicefrac{1}{2}}(\Sigma_1^{-1} - \Sigma_2^{-1})\right\|^2_{F} \leq \\ \leq
        2\|\Sigma_1^{-1} - \Sigma_*^{-1}\|_o^2 W^2_2(\Sigma_1, \Sigma_2) 
            + 2\|\Sigma_2\|_o \|\Sigma_1^{-1} - \Sigma_2^{-1}\|^2_{F} 
    \end{multline}
    where $\|\cdot\|_o$ and $\| \cdot \|_F$ stand for operator and Frobenius norm, respectively.

    Considering the quantity $\|\Sigma_1^{-1} - \Sigma_2^{-1}\|_F$ as $\|\Sigma_1 - \Sigma_2\|_F \to 0$.
    \begin{multline*}
        \Sigma_2^{-1} = 
        \left[\Sigma_1 - (\Sigma_1 - \Sigma_2) \right]^{-1} =
        \Sigma^{-1}_1\left[\Id - \underbrace{(\Sigma_1 - \Sigma_2)\Sigma^{-1}_1}_{:= \Delta \Sigma} \right]^{-1} = \\ =
        \Sigma^{-1}_1\left[\Id + \Delta\Sigma + \Delta\Sigma^2 + \dots \right] = 
        \Sigma^{-1}_1\left[\Id + \Delta\Sigma + R \right]
    \end{multline*}
    
    In order to bound the remainder of the series $R$, $\Sigma_1$ and $\Sigma_2$ should be sufficiently close.
    Let us find the exact condition before we proceed.
    Since $\Sigma_1,\ \Sigma_2$ are both symmetric
    \begin{multline*}
        \|\Delta \Sigma\|^2_F 
        = \|(\Sigma_1 - \Sigma_2)\Sigma^{-1}_1\|_F^2 = \\
        = \Tr\left(\Sigma^{-1}_1(\Sigma_1 - \Sigma_2)^2\Sigma^{-1}_1\right)
        = \Tr\left( (\Sigma^{-1}_1(\Sigma_1 - \Sigma_2))^T (\Sigma^{-1}_1(\Sigma_1 - \Sigma_2))\right) = \\
        = \|\Sigma^{-1}_1(\Sigma_1 - \Sigma_2) \|^2_F
    \end{multline*}
    where symmetry of the matrices and the cyclic property of the trace was used.
    Using~\eqref{eq:lambda_W2_bound} and~\eqref{eq:w2_frobenius_bound}
    \begin{multline*}
        \| \Delta \Sigma\|_F = 
        \| \Sigma^{-1}_1(\Sigma_1 - \Sigma_2) \|_F \leq \\
        \leq \|\Sigma^{-1}_1 \|_o \|\Sigma_1 - \Sigma_2 \|_F 
        \leq \frac{2\sqrt{2}(\sqrt{\lambda_1^*} + 2r)}{(\sqrt{\lambda_d^*} - r)^2}W_2(\Sigma_1, \Sigma_2) \leq \\
        \leq \frac{8\sqrt{2}(\sqrt{\lambda_1^*} + \sqrt{\lambda_d^*})}{\lambda^*_d}W_2(\Sigma_1, \Sigma_2)
    \end{multline*}
    since $r$ has to be less than $\frac{1}{2}\sqrt{\lambda_d^*}$.
    Thus, if
    \begin{equation*}
        W_2(\Sigma_1, \Sigma_2) 
        \leq \frac{1}{16\sqrt{2}\left( 1 + \sqrt{\frac{\lambda^*_1}{\lambda^*_d}}\right)}\sqrt{\lambda_d^*}
        \implies \|\Delta\Sigma\|_F \leq \frac{1}{2}
    \end{equation*}
    If $\|\Delta\Sigma\| \leq \frac{1}{2}$, the remainder of the series $R$ can be estimated as
    \begin{multline*}
        \| R\|_F = \left\|\Delta\Sigma^2 + \Delta\Sigma^3 + \dots \right\|_F \leq 
        \|\Delta\Sigma \|_F^2 + \|\Delta\Sigma \|_F^3 + \dots =\\= 
        \|\Delta\Sigma \|_F^2\left(1 + \|\Delta\Sigma \|_F + \|\Delta\Sigma \|_F^2 + \dots\right) = 
        \frac{\|\Delta\Sigma \|_F^2}{1 - \|\Delta\Sigma \|_F} \leq 2\|\Delta\Sigma \|_F^2 
    \end{multline*}
    Thus, we constrain the radius of the ball $r$ to be
    \begin{equation*}
        r 
        \leq \frac{1}{16\sqrt{2}\left( 1 + \sqrt{\frac{\lambda^*_1}{\lambda^*_d}}\right)}\sqrt{\lambda_d^*}
        \leq \frac{1}{32\sqrt{2}}\sqrt{\lambda^*_d}
    \end{equation*}
    Since $\frac{1}{32\sqrt{2}} < \frac{1}{2}$, the usage of~\eqref{eq:w2_frobenius_bound} was justified.
    Using this relation, we can get
    \begin{multline}\label{eq:taylor_2nd_ord}
        \|\Sigma_1^{-1} - \Sigma_2^{-1} \|^2_F = 
        \|\Sigma_1^{-1}(\Sigma_1 - \Sigma_2)\Sigma_1^{-1} + \Sigma_1^{-1}R\|^2_F \leq \\ \leq
        \| \Sigma_1^{-1}\|^2_o \|\Delta\Sigma + R\|^2_F \leq 
        2\|\Sigma_1^{-1}\|^2_o (\|\Delta\Sigma\|_F^2 + \|R\|_F^2) \leq \\ \leq
        2\|\Sigma_1^{-1}\|^2_o (\|\Delta\Sigma\|_F^2 + 4\|\Delta\Sigma\|_F^4) \leq 
        4\|\Sigma_1^{-1}\|^2_o \|\Delta\Sigma\|_F^2
    \end{multline} 
    Thus
    \[
        \|\Sigma_1^{-1} - \Sigma_2^{-1} \|^2_F \leq 
        \|\Sigma_1^{-1}\|_o^4 \|\Sigma_1 - \Sigma_2 \|^2_F \leq
        \frac{{8(\sqrt{\lambda^*_1} + 2r)^2} }{(\sqrt{\lambda^*_d} - r)^{8}} W^2_2(\Sigma_1 , \Sigma_2)
    \]
    We can finish the bound in \eqref{eq:lipschitz_gradient_start}
    \begin{multline}\label{eq:lipschitz_gradient_fin}
        \int \|\partial_W \cE(\Sigma_1, x) - \partial_W \cE(\Sigma_2, y) \|^2  \dd \gamma_{\text{opt}} \leq \\ \leq
        \left(2\|\Sigma_1^{-1} - \Sigma_*^{-1}\|_o^2 
            + 2\|\Sigma_2\|_o \frac{{8(\sqrt{\lambda^*_1} + 2r)^2} }{(\sqrt{\lambda^*_d} - r)^{8}} \right)W^2_2(\Sigma_1, \Sigma_2)  \leq \\ \leq
        \frac{4}{(\sqrt{\lambda^*_d} - r)^4} \left( 1
            + 4\frac{{(\sqrt{\lambda^*_1} + 2r)^4} }{(\sqrt{\lambda^*_d} - r)^{4}} \right)W^2_2(\Sigma_1, \Sigma_2)    
        \end{multline}
    \end{proof}

    \subsection{Derivation  of~\autoref{prop:projection_formula}}
    \begin{proof}\label{app:derivation_projection_formula}
    We extend scalar product~\eqref{eq:BW_scalar_prod} to the whole $\mathbb{R}^{d \times d}$ as $\langle U, V\rangle_\Sigma := \Tr(U^T\Sigma V)$ and pose the projection as the constrained minimization problem:
    \begin{align}
        \Pi_\Sigma(V) = &\arg\min_X \|X - V\|^2_\Sigma \\
            \text{s.t. }& X = X^T.
    \end{align}
    Writing out the Lagrangian and setting its derivatives to zero, one gets
    \begin{align*}
        L(X, \lambda) &= \Tr((X - V)^T\Sigma (X - V)) + \langle\lambda, X - X^T\rangle_F, \\
        \frac{\partial L}{\partial X} = 0 & \Rightarrow 
            2\Sigma X - 2\Sigma V = \lambda^T - \lambda, \\
        \frac{\partial L}{\partial \lambda} = 0 & \Rightarrow 
            X = X^T.
    \end{align*}
    Averaging the first equation and its transpose, with $X=X^T$, gives 
    \begin{equation*}
        \Sigma X + X \Sigma = \Sigma V + V^T \Sigma,
    \end{equation*}
    and setting $V = UT^{-1}_{01}$ yields \autoref{eq:BW_projection}.
    The unique solution exists if the eigenvalues of $\Sigma$ do not add up to zero:  
    $$
        \lambda_i + \lambda_j > 0\quad \forall (i,\ j),
    $$
    which always holds in our case since all the eigenvalues are positive. 
    See also~\cite{gajic2008lyapunov} for the overview of the numerical approaches. 
    \end{proof}
    \subsection{Proof of \autoref{them:Main}}
    \label{app:thm_main}

    \begin{lemma}[cf. Lemma 8 in \cite{li2023riemannian}]\label{lemm:boundonUpdate}
    Let $x_1,\dots,x_k\in \cB_{\mathcal N_0^d}(\Sigma_*,\frac{1}{2} r_{\Sigma_*})$. Then there is $M_1$ such that
    $$\|X_k\Gamma_k\|+\|r_k\|\leq M_1\sum_{j=0}^k\|r_{k-j}\|.$$
    
    \end{lemma}
    \begin{proof}
            We observe that $$\|\rExp_{x_\ell}^{-1}(x_{\ell+1})\| = W_2(x_\ell,x_{\ell+1})$$ by optimality of $\rExp^{-1}$. 
        Then by the Lipschitz assumption in \autoref{asmp:F_lip}, we have
        \begin{align*}
            \|\Delta x_{k-j}^{(k)}\| &= \|\mathcal T_{x_{k-1}}^{x_k}\cdots \mathcal T_{x_{x-j}}^{x_{k-j+1}}\rExp_{x_{k-j}}^{-1}(x_{k-j+1})\|\\
            &\leq M^j\|\rExp_{x_{k-j}}^{-1}(x_{k-j+1})\| \\
            &= M^jW_2(x_{k-j},x_{k-j+1})\\
            &\leq \frac{M^j}{L_1}\|r_{k-j+1}-\mathcal P_{x_{k-j}}^{x_{k-j+1}}r_{k-j}\|.
        \end{align*}

        It holds 
        \begin{align*}
            \|\sum_{j=1}^i\Delta x_{k-j}^{(k)}\|&\leq \sum_{k=1}^i\|\Delta x_{k-j}^{(k)}\|\leq \sum_{j=1}^i \frac{M^j}{L_1}(\|r_{k-j+1}\|+\|r_{k-1}\|)\\
            & \leq\frac M{L_1}\|r_{k}\|+\sum_{j=2}^i\frac{M^{{j-1}}(M+1)}{L_1}\|r_{k-j+1}\|+\frac{M^i}{L_1}\|r_{k-i}\|.
        \end{align*}
        Observe $ 2\frac{\max\{M^m,M\}}{L_1}\geq \frac 1{L_1}\max\{ M^m+M^{m-1}, M^2+M \}$. By \autoref{asmp:bounded_least_square} we have $\|\Gamma_k\|\leq M_\Gamma$. Hence,
        \begin{align*}
            \|X_k\Gamma_k\|+\|r_k\|&\leq \|\Gamma_k\|\sum_{j=1}^k\|\Delta x_{k-j}^{(k)}\|+\|r_k\|\\
            &\leq 2M_\Gamma \frac{\max\{M^m,M\}}{L_1}(\|r_k\|+\sum_{j=2}^k\|r_{k-j+1} \| + \|r_0\|)+\|r_k\|\\
            &\leq (2M_\Gamma \frac{\max\{M^m,M\}}{L_1}+1)\sum_{j=0}^k \|r_{k-j}\|.
        \end{align*}
\end{proof}
\begin{lemma}[cf. Proposition 2 in \cite{li2023riemannian}]\label{lemm:boundonUpdate2}
Let $x_1,\dots,x_k\in \cB_{\mathcal N_0^d}(\Sigma_*,\frac{r_{\Sigma_*}}{L_2(mM_1+\beta+1)})$. Let $z_k^i = \sum_{j=1}^i\Delta x_{k-j}^{(k)}$ and $y_k^i = \rExp_{x_k}(-z_k^i)$. Then there is $M_2,M_3>0$ such that
$$\|\mathcal P_{y_k^i}^{x_k}F(y_k^i)-\mathcal T_{x_{k-1}}^{x_k}\mathcal T_{x_{k-2}}^{x_{k-1}}\cdots\mathcal T_{x_{k-i}}^{x_{k-i+1}}F(x_{k-i})\|\leq M_2\sum_{j=0}^i\|r_{k-j}\|^2.$$
Furthermore, for $w_{k}^i = \sum_{j=0}^i\gamma_j^k\Delta x_{k-j}^{(k)}$ and $v_k^i = \rExp_{x_k}(-w_k^i)$ it holds
$$\|\mathcal P_{v_k^i}^{x_k}F(v_k^i)-\mathcal P_{v_k^{i-1}}^{x_k}F(v_k^{i-1})-\gamma_i^k(\mathcal P_{y_k^i}^{x_k}F(y_k^i)-\mathcal P_{yk^{i-1}}^{x_k}F(y_k^{i-1}))\|\leq M_3\sum_{j=0}^i\|r_{k-j}\|^2.$$
    
\end{lemma}
    

Now let us prove \autoref{them:Main}
        \begin{proof}\label{proof:Main}
            First we assume there are some $x_1,\dots,x_k\in \cB_{\mathcal N_0^d}(\Sigma_*,\tilde r)$ with $\tilde r< \frac{r_{\Sigma_*}}{2L_2(mM_1+\beta+1)}$.
        Define
        $$\begin{cases}
            \bar x_k = \rExp_{x_k}(-X_k\Gamma_k)\\
             x_{k+1} = \rExp_{x_k}(-X_k\Gamma_k+\beta_k\bar{r_k})\\
            \tilde x_{k+1} = \rExp_{x_k}(-X_k\Gamma_k+\bar { r_k})   
        \end{cases}$$
        with $\bar r_k = r_k - R_k\Gamma_r$. Then $\|\bar r_k\|\leq \|r_k\|$. 

        We observe \begin{align*}
            \|F(x_{k+1})\|\leq &L_2 W_2(x_{k+1},x_*)\\
            \leq & L_2(W_2(x_{k+1},x_k)+W_2(x_k,x_*))\\
            \leq& L_2(W_2(x_{k+1},x_k)+\tilde r)\\
            \leq& L_2(\|X_k\Gamma_k+\beta\bar r_k\|+\tilde r)\\
            \leq & L_2(M_1 (\sum_{i=0}^{\min\{k,m\}}\|r_{k-i}\|+\beta\|r_k\|)+\tilde r)\\
            \leq & L_2( M_1m\tilde r+\beta\tilde r+\tilde r)\\
            \leq &L_2(M_1m+\beta+1)\tilde r\leq r_{\Sigma_*}
        \end{align*}
Similarly, $W_2(\bar x_k,x^*),W_2(\tilde x_k,x^*)\leq L_2(M_1m+\beta+1)\tilde r. $
        
        Hence, with \autoref{lemm:boundonUpdate}  we get 
        \begin{align*}
        \| r_{k+1}\| =& W_2(G(x_{k+1}),x_{k+1})  \\
        \leq & W_2(G(x_{k+1}),G(\bar x_k))+W_2(G(\bar x_k),x_{k+1}).
        \end{align*}
        By assumption,
        $$W_2(G(x_{k+1}), G(\bar x_k))\leq \kappa W_2(\rExp_{x_k}(X_k\Gamma_k),\rExp_{x_k}(X_k\Gamma_k+\beta_k\bar { r_k}))  \leq\kappa \|\beta_k\bar{ r_k}\|.$$
        Now
        $$W_2(x_{k+1}, G(\bar x_k))\leq W_2(G(\bar x_k),\tilde x_{k+1})+W_2(\tilde x_{k+1},x_{k+1}).$$
        By assumption 
        $$W_2(\tilde x_{k+1},x_{k+1})\leq \|X_k\Gamma_k-\beta_k\bar { r_k} - X_k\Gamma_k+\bar{ r_k}\|=(1-\beta_k)\|\bar { r_k}\|.$$
        Also
        $$W_2(G(\bar x_k),\tilde x_{k+1})\leq W_2(G(\bar x_k),\rExp_{\bar x_k}(\mathcal T_{x_k}^{\bar x_k} \bar { r_k}))+W_2(\rExp_{\bar x_k}(\mathcal T_{x_k}^{\bar x_k} \bar { r_k}), \tilde x_{k+1}).$$
        By Lemma 5 in \cite{li2023riemannian} there is a constant $c_0$ depending on the sectional curvature such that $$W_2(\rExp_{\bar x_k}(\mathcal T_{x_k}^{\bar x_k} \bar { r_k}), \tilde x_{k+1})\leq c_0\underbrace{\min\{\|X_k\Gamma_k\|,\|\bar r_k\|\}}_{\leq \tilde r}(\|X_k\Gamma_k\|+\|\bar r_k\|)^2.$$
       With \autoref{lemm:boundonUpdate} we get $$(\|X_k\Gamma_k\|+\|\bar r_k\|)^2\leq mM_1^2 \sum_{j=0}^k\|r_{k-j}\|^2.$$
        Furthermore,
        $$W_2(G(\bar x_k),\rExp_{\bar x_k}(\mathcal T_{x_k}^{\bar x_k} \bar { r_k})) =W_2(\rExp_{\bar x_k}(-F(\bar r_k)),\rExp_{\bar x_k}(\mathcal T_{x_k}^{\bar x_k}\bar { r_k})) \leq \|-F(\bar x_k)-\mathcal T_{x_k}^{\bar x_k}\bar{ r_k}\|.$$
        Defining $w_k^i = \sum_{j=0}^i \gamma_j^k \Delta x_{k-j}^{(k)}$, $z_{j}^k =\sum_{j=0}^i  \Delta x_{k-j}^{(k)}, v_k^i=\rExp_{x_k}(-w_k^i)$ and $y_k^i = \rExp_{x_k}(-z_k^i) $ we get by \autoref{lemm:boundonUpdate2}
        \begin{align*}
            \|-F(\bar x_k)-\mathcal T_{x_k}^{\bar x_k}\bar{ r_k}\| & = \|\mathcal T^{x_k}_{\bar x_k}F(\bar x_k)+\bar{ r_k}\|\\
            &=\|\mathcal T^{x_k}_{\bar x_k}F(\bar x_k)-F(x_k)-\sum_{i=1}^k\gamma_i^k\Delta r_{k-i}\|\\
            &\leq\underbrace{\|\mathcal P_{v_k^i}^{x_k}F(v_k^i)-\mathcal P_{v_k^{i-1}}^{x_k}F(v_k^{i-1})-\gamma_i^k(\mathcal P_{y_k^i}^{x_k}F(y_k^i)-\mathcal P_{yk^{i-1}}^{x_k}F(y_k^{i-1}))\|}_{\leq M_3\sum_{j=0}^i\|r_{k-j}\|^2}\\
            &\quad +\underbrace{\sum_{i=1}^{k}|\gamma_i^k|\|-\Delta r_{k-i} + \mathcal P_{y_i^k}^{x_k}F(y_i^k)-\mathcal P_{y_k^{i-1}}^{x_k} F(y_k^{i-1})\|}_{\substack{\leq M_\Gamma(\|\mathcal T_{x_{k-1}}^{x_k}\cdots \mathcal T_{x_{k-i+1}}^{x_{k-i+2}} r_{k-i+1} +\mathcal P_{{y_k}^{i-1}}^{x_k} F(y_k^{i-1})\|\\+\|\mathcal T_{x_{k-1}}^{x_k}\cdots \mathcal T_{x_{k-i}}^{x_{k-i+1}} r_{k-i} +\mathcal P_{{y_k}^{i}}^{x_k} F(y_k^{i})\|)}}\\
            &\leq (M_3+2M_\Gamma M_2)\sum_{i=0}^k \|r_{k-i}\|^2.
        \end{align*}
        All-together, we get
        $$ \|r_{k+1}\|\leq (1-\beta_k+\kappa\beta_k)\|\bar r_k\|+(c_0\tilde rm M_1^2+M_3+2M_\Gamma M_2)\sum_{i=0}^k\|r_{k-i}\|^2.$$
        
        Now we want to show that the history produced by the algorithm fulfills these arguments. To that end we use induction.
        Define $\hat r=\min\{\frac{r_{\Sigma_*}}{L_2(mM_1+\beta+1)},\frac{L_1-L_2+(1-\kappa)\beta L_2}{MmL_2^2}\}$.
        If $W_2(x_0,x_*)< \frac{\hat r}{1+L_2}$ we get 
        $$\|F(x_0)\|\leq L_2 W_2(x_0,x_*)\leq \frac{L_2}{1+L_2}\hat r$$ and 
        \begin{align*}
            W_2(x_1,x_*)&\leq W_2(x_1,x_0)+W_2(x_0,x_*)\\
            &\leq \underbrace{\|\rExp^{-1}_{x_0}(x_1)\|}_{=\|r_0\|}+\frac 1 {1+L_2} \hat r\\
            &\leq \hat r.
        \end{align*}
        Assume now that the assumption of the theorem is true for all $j\leq k$. Then by above considerations we have 
        \begin{align*}
            \|r_{k+1}\|&\leq (1-\beta_k+\kappa\beta_k)\|r_k\|+M\sum_{j=0}^k\|r_{k-j}\|^2\\
            &\leq (1-\beta_k+\kappa\beta_k)L_2\hat r+MmL_2^2\hat r^2\\
            &\leq L_1 \hat r.
        \end{align*}
        Then we have $$W_2(x_{k+1},x_*)\leq \frac 1 {L_1} \|r_{k+1}\|\leq  \hat r$$ and the claim follows by induction.
    \end{proof}

\section{Additional numeric results}
    For all the mentioned problems, a comprehensive set of numerical experiments has been conducted.
    We study the performance of the method depending on both the problem settings and the hyperparameters of the numeric solver.
    The selection of the former will be described for each problem individually.
    As for the method hyperparameters, for the BWRAM method we vary the relaxation parameter $\beta_k$, the maximal number of history vectors $m$, and the regularization parameter $l_{\infty, max}$. 
    The number of historic vectors varies in the range from $1$ to $15$.
    In the current work, these are constant during the iterations.
    In principle, adaptive strategies for restarting the iteration~\cite{novak2023adaptive,wei2023convergence} or selection of the relaxation parameter~\cite{irons1969version}, but these strategies are outside of the scope of the current work.

    For the Ornstein-Uhlenbeck process and $\operatorname{KL}$ minimization, we consider multiple invariant distributions with covariance matrix being {i.i.d} realizations of the following random variable: 
    \begin{align*}
        \Sigma_* &=  Q^T D Q, \quad
        D = \operatorname{diag}\left(\sigma_{\min} + \frac{i - 1}{d - 1}(\sigma_{\max} - \sigma_{\min}) \right),\quad
        Q\text{ i.i.d. random orthogonal matrix}.
    \end{align*}
    We fix $\sigma_\textrm{min}$ and vary $d$ and $\sigma_\textrm{max}$.
    For each pair of $d,\ \sigma_\textrm{max}$, we run the Picard iteration and Anderson acceleration with different lengths of history $m$.
    The acceleration measured as $\frac{N_\text{Picard}}{N_\text{AA}}$ is averaged over several independent realizations of $Q$.

    For the Ornstein-Uhlenbeck problem, the relaxation parameter and regularization parameter are set to be $\beta = 1.0$ and $l_{\infty,max} = 1.0$.
    In \autoref{tab:OU_acceleration}, one can see the maximal mean acceleration (averaged over $6$ independent runs) of BWRAM, compared to Picard.

    For the $\operatorname{KL}$ minimization problem, where the operator takes form
    \[
        G(\Sigma) = \rExp_\Sigma(-h\partial_W \operatorname{KL}(\Sigma| \Sigma_*)),
    \]
    the scaling parameter is set to $h=0.3$, and the values of $d,\ \sigma_{\textrm{min}},\ \sigma_{\textrm{max}}$ are the same as in the Ornstein-Uhlenbeck problem. 
    As for the hyperparameters of BWRAM, the regularization parameter takes values  $l_{\inf,\max} \in  \{ 1.5, 10.0\}$.
    In case of fixed-point problems, with an operator defined by the Wasserstein gradient of some functional $\cE$ as $\rExp(-h\partial_W \cE)$, the optimal stepsize parameter $h$ may have different values for the basic Picard method and for RAM.
    One can compensate for that by choosing the relaxation $\beta_k$ factor smaller or larger than one.
    We demonstrate it by varying the relaxation factor in the $\operatorname{KL}$ problem in the set $\beta \in \{0.9, 1.0, 1.2\}$.
    Maximal acceleration is achieved for $\beta = 1.2$ in most cases.

    We also compare our method to established methods of Riemannian optimization.
    In particular, we consider Riemannian gradient descent (RGD) with backtracking line search, and Riemannian conjugate gradient descent, provided by the~\texttt{pymanopt}~\cite{townsend2016pymanopt} package.
    The results of the experiments are presented in \autoref{tab:KL_n_steps_by_method}.
    Here, the number of iterations is first averaged over $n = 6$ runs for random $\Sigma_*$, and then the minimal one is chosen among different values of the hyperparameters.
    The best result for each problem is marked bold.
    \begin{table}[ht]
            \centering
            \caption{The mean number of steps for the $\operatorname{KL}$ minimization problem for each method.}
            \begin{tabular}{llrrrr}
\toprule
 &  & Picard & RGD & BWRAM & RCG \\
$d$ & $\sigma_{max}$ &  &  &  &  \\
\midrule
\multirow[c]{3}{*}{4} & 5.0 & 45.0 & 53.0 & \bfseries 11.0 & 25.0 \\
 & 10.0 & 95.0 & 113.0 & \bfseries 15.0 & 42.0 \\
 & 20.0 & 196.0 & 162.0 & \bfseries 18.2 & 56.0 \\
\cline{1-6}
\multirow[c]{3}{*}{8} & 5.0 & 45.0 & 39.0 & \bfseries 14.0 & 33.0 \\
 & 10.0 & 96.0 & 124.0 & \bfseries 22.8 & 53.0 \\
 & 20.0 & 197.0 & 243.0 & \bfseries 38.6 & 68.0 \\
\cline{1-6}
\multirow[c]{3}{*}{16} & 5.0 & 46.0 & 64.0 & \bfseries 14.6 & 29.0 \\
 & 10.0 & 97.0 & 109.0 & \bfseries 24.2 & 39.0 \\
 & 20.0 & 200.0 & 230.0 & \bfseries 41.6 & 69.0 \\
\cline{1-6}
\multirow[c]{3}{*}{32} & 5.0 & 47.0 & 51.0 & \bfseries 15.0 & 31.0 \\
 & 10.0 & 100.0 & 110.0 & \bfseries 22.8 & 42.0 \\
 & 20.0 & 206.0 & 265.0 & \bfseries 42.6 & 66.0 \\
\cline{1-6}
\multirow[c]{3}{*}{64} & 5.0 & 49.0 & 58.0 & \bfseries 15.8 & 31.0 \\
 & 10.0 & 103.0 & 117.0 & \bfseries 22.4 & 57.0 \\
 & 20.0 & 212.0 & 233.0 & \bfseries 42.8 & 68.0 \\
\cline{1-6}
\multirow[c]{3}{*}{128} & 5.0 & 51.0 & 51.0 & \bfseries 17.0 & 30.0 \\
 & 10.0 & 107.0 & 124.0 & \bfseries 23.6 & 54.0 \\
 & 20.0 & 220.0 & 273.0 & \bfseries 42.4 & 74.0 \\
\cline{1-6}
\bottomrule
\end{tabular}

            \medskip
            \label{tab:KL_n_steps_by_method}
    \end{table}

    In case of the averaging problems (Barycenter, Entropic Barycenter, Median), the parameters of the invariant distributions are the dimension $d$ and the number of distributions $n_\sigma$, taking values from  sets $d \in \{3, 5, 10, 20, 50 \}$ and $\{3, 5, 10, 20  \}$, respectively.
    The number of iterations for each method is averaged for $5$ random initializations of the distributions $\Sigma_i$, which are i.i.d., drawn from the Wishart distribution $W(\Id, d)$.
    The entropic regularization parameter in the Entropic Barycenter problem is set to $\gamma = 0.001$ and the smoothing parameter in the Median problem is set to $\varepsilon = 0.001$.
    In all three experiments, the iteration proceeds until the cost converges to the minimum up to the tolerance of $\epsilon = 10^{-10}$.
    As for the hyperparameters of the method, for Barycenters and Entropic barycenters, they are the same as for the $\operatorname{KL}$ problem.
    For the Median problem, a broader range of relaxations and regularizations was considered, namely $l_{\inf,\max} \in \{0.1,\ 5.0,\ 15.0 \}$ and $\beta_k \in \{1.0,\ 5.0,\ 10.0 \}$.
    The mean number of steps taken by each method with the best set of hyperparameters, is given in~\autoref{tab:Averaging_n_steps_by_method}. 
    We can see that BWRAM always provided acceleration compared to plain Picard iteration, in fact more that an order of magnitude in certain cases.
    Riemannian methods are also outperformed in most cases, the exception being the results for the Median problem in dimension $d = 50$, where RGD is superior.
    We note however, that the performance of BWRAM in that case is still comparable.

    \renewcommand{\arraystretch}{1.2} 
    \setlength{\tabcolsep}{3pt} 
%
    \begin{table}[htbp]
        \centering
        \caption{Mean number of iterations for the three averaging problems: Barycenter, Entropic Barycenter and Median. The best result for each pair of $(d,\ n_\sigma)$ is given in bold.} 
        \begin{tabular}{llrrrrrrrrrrrr}
\toprule
 &  & \multicolumn{4}{r}{Barycenter} & \multicolumn{4}{r}{Entropic Barycenter} & \multicolumn{4}{r}{Median} \\
 &  & BWRAM & Picard & RCG & RGD & BWRAM & Picard & RCG & RGD & BWRAM & Picard & RCG & RGD \\
$d$ & $n_{\sigma}$ &  &  &  &  &  &  &  &  &  &  &  &  \\
\midrule
\multirow[c]{4}{*}{3} & 3 & \bfseries 5.4 & 6.2 & 31.6 & 14.4 & \bfseries 4.4 & 6.6 & 15.8 & 8.8 & \bfseries 12.8 & 27.0 & 31.0 & 30.3 \\
 & 5 & \bfseries 7.0 & 9.8 & 29.0 & 16.6 & \bfseries 5.8 & 10.6 & 12.2 & 9.0 & \bfseries 13.8 & 26.5 & 24.5 & 21.5 \\
 & 10 & \bfseries 6.8 & 8.6 & 28.6 & 15.8 & \bfseries 4.8 & 10.4 & 16.4 & 9.8 & \bfseries 13.4 & 23.3 & 22.4 & 15.8 \\
 & 20 & \bfseries 6.6 & 9.0 & 26.2 & 15.2 & \bfseries 5.0 & 9.2 & 13.4 & 10.0 & \bfseries 13.3 & 22.9 & 26.1 & 16.3 \\
\cline{1-14}
\multirow[c]{4}{*}{5} & 3 & \bfseries 7.4 & 9.8 & 32.4 & 18.0 & \bfseries 6.4 & 8.8 & 14.2 & 9.0 & \bfseries 23.2 & 72.1 & 32.5 & 33.3 \\
 & 5 & \bfseries 7.2 & 9.6 & 34.8 & 18.2 & \bfseries 5.6 & 8.2 & 15.8 & 9.8 & \bfseries 13.5 & 46.1 & 24.5 & 17.2 \\
 & 10 & \bfseries 7.4 & 10.6 & 38.4 & 18.8 & \bfseries 6.0 & 9.0 & 17.0 & 10.6 & \bfseries 11.2 & 41.8 & 21.3 & 16.6 \\
 & 20 & \bfseries 7.2 & 10.0 & 36.0 & 19.0 & \bfseries 5.0 & 9.4 & 18.4 & 10.8 & \bfseries 10.0 & 39.3 & 22.1 & 14.8 \\
\cline{1-14}
\multirow[c]{4}{*}{10} & 3 & \bfseries 8.6 & 13.6 & 26.4 & 19.2 & \bfseries 7.0 & 14.2 & 9.0 & 8.6 & \bfseries 11.9 & 120.9 & 30.8 & 25.2 \\
 & 5 & \bfseries 9.0 & 13.8 & 33.4 & 20.4 & \bfseries 6.2 & 13.2 & 14.0 & 11.2 & \bfseries 10.3 & 94.1 & 26.7 & 18.7 \\
 & 10 & \bfseries 8.8 & 13.8 & 40.6 & 22.6 & \bfseries 6.6 & 11.0 & 15.6 & 12.6 & \bfseries 10.4 & 89.7 & 22.7 & 15.6 \\
 & 20 & \bfseries 8.0 & 11.2 & 39.4 & 22.0 & \bfseries 5.4 & 11.4 & 16.2 & 11.0 & \bfseries 9.8 & 82.8 & 16.5 & 13.5 \\
\cline{1-14}
\multirow[c]{4}{*}{20} & 3 & \bfseries 10.6 & 20.6 & 34.6 & 23.8 & \bfseries 7.8 & 23.2 & 11.2 & 11.4 & \bfseries 11.8 & 245.7 & 28.5 & 23.9 \\
 & 5 & \bfseries 11.2 & 19.0 & 36.6 & 21.6 & \bfseries 7.0 & 23.8 & 12.2 & 11.0 & \bfseries 11.3 & 206.0 & 30.5 & 19.7 \\
 & 10 & \bfseries 10.0 & 17.0 & 44.6 & 22.8 & \bfseries 7.0 & 19.4 & 15.2 & 12.2 & \bfseries 10.8 & 192.8 & 27.5 & 18.5 \\
 & 20 & \bfseries 8.8 & 13.4 & 44.8 & 21.8 & \bfseries 6.8 & 15.8 & 18.0 & 12.4 & \bfseries 10.6 & 177.5 & 14.8 & 14.5 \\
\cline{1-14}
\multirow[c]{4}{*}{50} & 3 & \bfseries 13.4 & 26.4 & 37.2 & 30.8 & \bfseries 8.6 & 23.4 & 16.0 & 14.0 & 28.4 & 707.5 & 32.2 & \bfseries 28.0 \\
 & 5 & \bfseries 12.8 & 23.8 & 50.2 & 29.2 & \bfseries 8.6 & 28.8 & 21.8 & 14.6 & 26.5 & 601.3 & 35.1 & \bfseries 22.8 \\
 & 10 & \bfseries 11.2 & 18.6 & 44.4 & 27.4 & \bfseries 7.0 & 22.4 & 21.0 & 14.4 & \bfseries 25.5 & 506.9 & 44.3 & 25.8 \\
 & 20 & \bfseries 9.2 & 14.2 & 44.0 & 26.4 & \bfseries 6.4 & 15.8 & 19.8 & 13.8 & \bfseries 20.5 & 461.2 & 40.5 & 23.5 \\
\cline{1-14}
\bottomrule
\end{tabular}

        \label{tab:Averaging_n_steps_by_method}
    \end{table}



\end{document}